\documentclass{amsart}
\usepackage{amsmath, amsfonts, latexsym, array, amssymb, amscd }
\usepackage[all]{xypic}
\usepackage[all]{xy}
\usepackage{graphicx}
\usepackage{color}
\usepackage{enumitem}
\usepackage{epstopdf}
\usepackage{hyperref}
\usepackage{mathtools}
\usepackage{mathrsfs}
\usepackage{parskip}
\usepackage{accents}
\usepackage[svgnames]{xcolor}
\usepackage{caption}
\usepackage{subcaption}
\usepackage{tikz}
\usepackage{tikz-cd}
\usetikzlibrary{calc}
\usetikzlibrary{arrows.meta}

\numberwithin{equation}{section}
\numberwithin{figure}{section}

\newtheorem{theorem}{Theorem}[section]
\newtheorem*{theorem*}{Theorem}
\newtheorem{lemma}[theorem]{Lemma}

\newtheorem{proposition}[theorem]{Proposition}
\newtheorem*{proposition*}{Proposition}
\newtheorem{corollary}[theorem]{Corollary}
\newtheorem{thmx}{Theorem}

\theoremstyle{definition}
\newtheorem{definition}[theorem]{Definition}

\newtheorem{remark}[theorem]{Remark}

\def \Proj{\operatorname{Proj}}
\def \Int{\operatorname{Int}}

\usepackage{indentfirst}
\DeclareMathOperator{\diam}{diam}

\newcommand{\SSS}{\mathcal{S}}
\newcommand{\GGG}{\mathcal{G}}
\newcommand{\RR}{\mathcal{R}}

\newcommand{\BBB}{\mathcal{B}}

\newcommand{\HHH}{\mathcal{H}}

\newcommand{\DDD}{\mathcal{D}}

\newcommand{\WWW}{\mathcal{W}}

\newcommand{\R}{\mathbb{R}}
\newcommand{\Z}{\mathbb{Z}}
\newcommand{\abs}[1]{\left\lvert#1\right\rvert}

\newcommand{\FFF}{\mathcal{F}}

\def\CAT{\operatorname{CAT}}
\newcommand{\bX}{\partial X}

\newcommand{\wguu}{\mathcal{W}^{\mathrm{uu}}_{\HHH}}
\newcommand{\wgu}{\mathcal{W}^{\mathrm{u}}_{\HHH}}
\newcommand{\wgss}{\mathcal{W}^{\mathrm{ss}}_{\HHH}}
\newcommand{\wgs}{\mathcal{W}^{\mathrm{s}}_{\HHH}}
\newcommand{\wuu}{W^{\mathrm{uu}}}
\newcommand{\wu}{W^{\mathrm{u}}}
\newcommand{\wss}{W^{\mathrm{ss}}}
\newcommand{\ws}{W^{\mathrm{s}}}

\newcommand{\ang}[2]{\langle#1,#2\rangle}

\newcommand{\ul}{H\"older\textsubscript{ul}}

\title[$C^\alpha$ metric Anosov flows]
{H\"older regularity for metric Anosov flows}
\author{Caleb~Dilsavor and Daniel~J.~ Thompson} 
\date{\today}
\thanks{D.T. was partially supported by NSF grant DMS-$1954463$ and DMS-$2349915$.}
\address{D.~J.~Thompson, Department of Mathematics, The Ohio State University, Columbus, OH 43210, \emph{E-mail address:} \tt{thompson@math.osu.edu}}
\address{C. Dilsavor, Department of Mathematics, Northwestern University, 2033 Sheridan Road, Evanston, IL 60208, USA}
\email{caleb.dilsavor@northwestern.edu}
\begin{document}
\maketitle

\begin{abstract} 
We define a class of flows on compact metric spaces that we call \emph{$C^\alpha$ metric Anosov flows}, $0 < \alpha < 1$ or $\alpha = \mathrm{Lip}$. We give verifiable criteria for a metric Anosov flow to be $C^\alpha$. We show that $C^\alpha$ metric Anosov flows have H\"older continuous local product structure, and in particular H\"older continuous holonomies. We deduce that topological conjugacies between $C^\alpha$ metric Anosov flows are automatically H\"older continuous. We show that any $C^\alpha$ metric Anosov flow admits a strong (i.e. H\"older continuous) Markov coding. This yields strong ergodic properties for equilibrium states for H\"older potentials for transitive $C^\alpha$ metric Anosov flows. The regularity of the local product structure is proved in a non-compact setting which applies for the geodesic flow on the universal cover of a locally $\CAT(-1)$ space. This allows us to obtain H\"older regularity results for the holonomy maps for pinched negative curvature manifolds with respect to the Sasaki metric. This contrasts with the result of Ballmann, Brin and Burns that there are examples in this class for which the dynamical distributions are not H\"older. 
\end{abstract}

\section{Introduction}

For dynamical systems with some hyperbolicity, the regularity of dynamically-defined sets and operations is a central issue. Foundational references for smooth maps and flows include \cite{PSW, HW, bH94, KH, SS92, HH77} and \cite[\S 9]{FH19}. The paper \cite{bH94} contains an excellent overview of the classic developments in this area up to the early 90's.  We develop a framework in which H\"older regularity of these dynamically-defined sets and operations can be studied using relatively elementary techniques and without requiring smooth structure. Metric Anosov flows, or Smale flows, on compact spaces were introduced by Pollicott in \cite{Pollicott} as a non-smooth abstraction of Axiom A flow by assuming only a list of axioms for a dynamical bracket operation.  This class has proved useful in recent years, particularly for the theory of pressure metrics for Anosov representations developed by Bridgeman, Canary, Labourie and Sambarino \cite{BCLS}.  Another natural class of metric Anosov flows is provided by the geodesic flow for a compact locally $\CAT(-1)$ spaces \cite{CLT19}.
	
	A drawback of the original definition of metric Anosov flows is that there is no reason that they should have any nice H\"older properties. It has been a challenge to find natural regularity criteria to build into this framework.  Pollicott showed that metric Anosov flows admit a Markov coding (i.e. a description by a suspension flow over a shift a finite type) by extending Bowen's construction for Axiom A flows \cite{rB73}. However, this coding might lack H\"older regularity.  We say a Markov coding is a \emph{strong Markov coding} if the suspension flow has a H\"older roof function and the projection to the flow is H\"older. This issue was addressed in Constantine, Lafont and the second-named author in \cite{CLT19} with a focus on examples arising from geometry. They gave criteria to obtain a strong Markov coding, and verified this criteria for geodesic flows on $\CAT(-1)$ spaces and for projective Anosov representations. 
	
	The approach in \cite{CLT19} gives abstract criteria for the existence of a strong Markov coding. However, the arguments to verify these criteria given in \cite{CLT19} were ad hoc and were constructed for specific classes of examples. These ad hoc arguments were heavily geometric and do not apply to more general metric Anosov flows. 
	
	In this paper, we go further and provide a natural framework to study H\"older regularity for metric Anosov flows by defining the class of $C^\alpha$ metric Anosov flows. Our definitions are set up differently from the previous literature and they provide a natural framework for verifiable H\"older conditions on the flow.   The known classes of examples of metric Anosov flows of interest can readily be shown to be $C^\alpha$ metric Anosov flows.  From our `minimal' natural H\"older assumptions at the start, we prove that all of the more complex dynamical sets and operations that we are interested in are H\"older. In particular, we show that $C^\alpha$ metric Anosov flows satisfy H\"older continuity of the Bowen bracket and thus the natural holonomy maps defined by the bracket operation are H\"older continuous. From this, we give a general argument to verify the criteria in \cite{CLT19} and thus we show that every $C^\alpha$ metric Anosov flow admit a strong Markov coding. This extends the results of \cite{CLT19} to a natural and more easily verifiable framework. When the flow is transitive, the strong Markov coding implies that $C^\alpha$ metric Anosov flows satisfy many stronger dynamical and statistical properties: for example,  the unique equilibrium state of a H\"older potential satisfies the CLT. 	
	We provide an alternative formulation of our definition in terms of laminations which is closer to the definitions in \cite{BCLS}. This may have advantages in some situations, and demonstrates the relationship between our framework and other formulations in the literature.  
	
	Our regularity proofs apply in some non-compact settings, and they have independent interest for geodesic flows in pinched negative curvature, where they contrast with the result of \cite{BBB} that H\"older regularity of the dynamical distributions does not hold without the further assumption that the derivatives of curvature are bounded.  In this setting, we show that the dynamical foliations have H\"older holonomies in a more general sense which allows a wide class of smooth transversals. To summarize the main points of this paper, we have:
	\vspace{2pt}
\begin{quote}
\emph{We provide a natural verifiable framework for metric Anosov flows on compact spaces to have H\"older continuity of: their local product structure and associated holonomy maps and their coding by a suspension flow over a Markov shift. This applies to all the examples of metric Anosov flows in the literature. An application of our methods is that the holonomy maps for the geodesic flow on any pinched negative curvature manifold are H\"older continuous without needing to assume that the derivatives of curvature are bounded.}
\end{quote}
\vspace{5pt}
The paper is structured as follows. 
In \S \ref{sec:defs}, we set up our background and notation. In \S \ref{sec:statement}, we state our results. In \S \ref{s.main}, we prove our main results on H\"older regularity of the Bowen bracket and holonomy maps.
In \S \ref{s.coding}, we prove Theorem \ref{Coding} concerning the regularity of
the coding. In \S \ref{s.geodesicflow}, we apply our results to geodesic flows. In \S \ref{subsec.otherdefinitions}, we give an equivalent formulation of our definition in terms of laminations and establish the relationships between our definition and others in the literature.

\section{Definitions} \label{sec:defs}

\subsection{General setup} Let $\FFF: Y \times \R \to Y$ be a continuous flow on a complete separable metric space $(Y,d)$, and we write
$f_t(y) = \FFF(y,t)$. For $\delta>0$, we write
\[
(Y\times Y)_{\delta}:= \{(x,y): x, y \in Y, d(x, y) \leq \delta\}.
\]
We sometimes refer to $(Y\times Y)_{\delta}$ informally as a \emph{neighborhood of the diagonal of $Y$}. We equip $(Y\times Y)_{\delta}$ with the metric given by taking the maximum distance of both its coordinates. 

The \emph{strong stable set} and \emph{strong unstable set} around $y \in Y$ are defined by
\begin{align*}
	\wss(y;d) &= \{ x \in Y \colon \lim_{t\to\infty}d(f_tx,f_ty) = 0\},\\
	\wuu(y;d) &= \{ x \in Y \colon \lim_{t\to-\infty}d(f_tx,f_ty) = 0\},
\end{align*}
respectively, and for $\epsilon > 0$ their local versions are
\begin{align*}
	\wss_\epsilon(y;d) &= \{x \in \wss(y;d) \colon d(f_tx,f_ty) \leq \epsilon \text{
	for } t\geq 0\},\\
	\wuu_\epsilon(y;d) &= \{x \in \wuu(y;d) \colon d(f_tx,f_ty) \leq \epsilon \text{
	for } t\leq 0\}.
\end{align*}
We define the global and local \emph{weak unstable sets} and \emph{weak stable sets} by
\[
	\ws(y;d) = \bigcup_{v\in \R} \wss(f_vy;d), \quad
	\wu(y;d) = \bigcup_{v\in\R}\wuu(f_vy;d)
,\]
\[
\ws_\epsilon(y;d) = \bigcup_{\abs{v} \leq \epsilon} \wss_\epsilon(f_vy;d), \quad 
\wu_\epsilon(y;d) = \bigcup_{\abs{v}\leq \epsilon}\wuu_\epsilon(f_vy;d). \]
When the metric $d$ is clear from the context, then we simply write $\wss(y)$, $\wuu(y)$, $\ws_\epsilon(y)$, and $\wu_\epsilon(y)$. Our definition of $C^\alpha$ metric Anosov flows assumes that the space $(Y,d)$ is assumed to be compact, but our regularity results have interest and applications beyond the compact situations and the above definitions make sense without a compactness assumption.

\subsection{Definition of $C^\alpha$ metric Anosov flows} 
Let $\FFF: Y \times \R \to Y$ be a continuous flow on a compact metric space $(Y,d)$.
\begin{definition}\label{def.cc}
	We say that $\FFF$ has \emph{canonical coordinates at scale $(\epsilon,\delta)$} if,
	for every $x,y \in Y$ with $d(x,y) \leq \delta$, there is a unique
	$v = v(x,y)$ with $\abs{v} \leq \epsilon$ such that
	\[
		\wss_{\epsilon}(x)\cap \wuu_{\epsilon}(f_{v}y) \neq \emptyset
	,\]
	and furthermore this intersection is a singleton.  We say that $\FFF$ has canonical coordinates if for every small enough $\epsilon > 0$, there
	exists $\delta > 0$ such
	say that $\FFF$ has canonical coordinates at scale $(\epsilon,\delta)$.
	We let $\ang{x}{y}$ be the unique element of this
	singleton when it is defined, and we call $\ang{x}{y}$ the \emph{Bowen bracket} of $x$ and $y$. We call the map $\ang{\cdot}{\cdot}: (Y \times Y)_\delta \to Y$ the  \emph{Bowen bracket} for $\FFF$.
\end{definition}
Note that if  $\FFF$ has canonical coordinates at scale $(\epsilon,\delta)$ and $\delta'\in(0, \delta)$, then   $\FFF$ has canonical coordinates at scale $(\epsilon,\delta')$. We emphasize that in the notation $\ang{x}{y}$, we use the convention that we are on the same strong stable as $x$ and the same weak unstable as $y$. This agrees with the most widely used convention in the literature\footnote{We note that some sources, including \cite{CLT19}, use the opposite convention.} 
\begin{definition}\label{def.Calpha}
	Let $Y$ be a compact metric space, and let $\FFF : Y\times \R \to Y$ be a continuous flow with
	canonical coordinates. Let $0 < \alpha \leq 1$, and suppose that there exist $\eta, k, \lambda >
	0$, $C \geq 1$ such that
	\begin{equation}\label{eq.HW}
		\begin{split}
			d(f_tx,f_ty) \leq Ce^{-\lambda t}d(x,y)^\alpha
			\quad (y \in Y, \: x \in \wss_{\eta}(y),\: t \geq 0),\\
			d(f_{-t}x',f_{-t}y) \leq Ce^{-\lambda t}d(x',y)^\alpha
			\quad (y \in Y, \: x' \in \wuu_\eta(y), \: t \geq 0),
		\end{split}
		\tag{HW}
	\end{equation}
	\begin{equation}\label{eq.HF}
		d(f_tx,f_ty)\leq Ce^{k\abs{t}}d(x,y)^\alpha \quad (x,y \in Y, \: d(x,y) \leq \eta, \: t \in \R),
		\tag{HF}
	\end{equation}
	\begin{equation}\label{eq.HS}
		\frac{1}{C}\abs{t}^{\frac{1}{\alpha}} \leq d(y,f_ty)\leq C\abs{t}^\alpha \quad 
		(y \in Y, \: \abs{t} \leq \eta).
		\tag{HS}
	\end{equation}
	For $0 < \alpha <1$, we say that $\FFF$ is a \emph{$C^\alpha$ metric Anosov flow}. When $\alpha = 1$, we say that $\FFF$ is a \emph{Lipschitz metric Anosov flow}.
	\end{definition}
Note that we must have $\lambda \leq k$, so we assume this. Informally, the properties \eqref{eq.HW}, \eqref{eq.HF}, and \eqref{eq.HS} are respectively H\"older contraction in the local strong stable and unstable leaves, H\"older continuity of the flow, and H\"older continuity of the speed. The definitions of  \eqref{eq.HW}, \eqref{eq.HF}, and \eqref{eq.HS} also apply when $Y$ is not assumed to be compact.

\begin{remark}Our definition works directly with the stable and unstable sets without assuming a priori that they have any topological structure. This differs from, and simplifies, the original presentation by Pollicott in \cite{Pollicott}. In \S \ref{subsec.otherdefinitions}, we prove that our definition can be formulated in terms of topological laminations. The definition of metric Anosov flows in \cite{BCLS} is given in that language and is equivalent to Pollicott's original definition. In particular, this shows that our definition of  $C^\alpha$ metric Anosov flows are consistent with Pollicott's definition as illustrated in Figure \ref{fig:metric-anosov}. The idea of giving topological definitions which capture the essential features of smooth hyperbolic dynamics was likely initiated by Bowen in his thesis and early work, in particular his definition of Axiom $A^\ast$ homeomorphisms in \cite{rB71}.
\end{remark}
\subsection{Uniform canonical coordinates for non-compact spaces}
Recall that a function $\Psi:Y_1 \to Y_2$ between metric spaces
$(Y_1,d_1)$ and $(Y_2,d_2)$ is \emph{uniformly locally H\"older}, which we abbreviate to
\ul{}, if there exist $A,\delta > 0$, $\gamma \in (0, 1]$ such that, for all $x,y\in Y_1$ with $d_1(x,y)
\leq \delta$, then
\[
	d_2(\Psi x, \Psi y) \leq A\, d_1(x,y)^\gamma.
\]
The constant $\gamma$ is called the \emph{exponent} of the \ul{}
continuity. We follow verbatim the definition of the canonical coordinates in the compact case. We emphasize that the content in the non-compact case is that the constants are chosen uniformly over the whole space.
\begin{definition}\label{def.ccu}
	The flow $\FFF$ has \emph{uniform canonical coordinates at scale $(\epsilon,\delta)$} if
	for every $x,y \in Y$ with $d(x,y) \leq \delta$, there is a unique
	$v = v(x,y)$ with $\abs{v} \leq \epsilon$ such that
	\[
		\wss_{\epsilon}(x)\cap \wuu_{\epsilon}(f_{v}y) \neq \emptyset
	,\]
	and this intersection is a singleton. If for every small enough $\epsilon > 0$, there
	is some $\delta > 0$ such that $\FFF$ has canonical coordinates at scale $(\epsilon,\delta)$, then we
	say that $\FFF$ has \emph{uniform canonical coordinates}.
	We let $\ang{x}{y}$ be the unique element of this
	singleton when it is defined, and we call $\ang{x}{y}$ the \emph{Bowen bracket} of $x$ and $y$.
\end{definition}
Note that if  $\FFF$ has uniform canonical coordinates at scale $(\epsilon,\delta)$ and $\delta'\in(0, \delta)$, then   $\FFF$ has uniform canonical coordinates at scale $(\epsilon,\delta')$.
\begin{remark}
Let $(M,\mathfrak{g})$ be a complete pinched negatively curved manifold. The geodesic flow on the universal cover $(\tilde M, \tilde{ \mathfrak{g}})$ has uniform canonical coordinates. This is the example we have in mind for our definition and we use this for our applications to geodesic flows. We emphasize that if $M$ has a cusp, then the `downstairs' geodesic flow over $(M,\mathfrak{g})$ does not have uniform canonical coordinates in the sense above. 
\end{remark}
\subsection{Holonomy maps}
Let $\FFF: Y \times \R \to Y$ be a continuous flow on a complete separable metric space $(Y,d)$ with uniform canonical coordinates at scale $(\epsilon, \delta)$. In particular, this applies when $(Y,d)$ is compact and has canonical coordinates at scale $(\epsilon, \delta)$. Given  $x,y \in Y$ with $d(x,y) \leq \delta$, the Bowen bracket satisfies
\[
\{\ang{x}{y}\} = \wss_\epsilon(x) \cap\wu_\epsilon(y).
\]
We use this to define the \emph{local $\wu$-holonomy maps with
$\wss$-transversals}, or simply the \emph{u-holonomy maps} on $\wss$.
\[
	h^u_{xy} : \wss_\delta(x) \to \wss_\epsilon(y), \qquad h^u_{xy}(x') = \ang{y}{x'} \in
	\wss_\epsilon(y)\cap \wu_\epsilon(x').
\] 
Similarly, we define the \emph{local $\wss$-holonomy maps with
$\wu$-transversals},  or simply \emph{ss-holonomy maps} on $\wu$, by
\[
	h^{ss}_{xy} : \wu_\delta(x) \to \wu_\epsilon(y), \qquad h^{ss}_{xy}(x') = \ang{x'}{y} \in \wss_\epsilon(x')
	\cap \wu_\epsilon(y).
\] 

\begin{figure}[ht]
  \centering
  \begin{tikzpicture}[>=Stealth,scale=1.15,
    leaf/.style={very thick},
    bg/.style={gray!45,thin},
    ss/.style={thick}]

  \foreach \a in {-0.575,0,0.575}{
    \draw[bg,smooth,samples=40,domain=-1.15:1.15,variable=\b]
      plot ({\a+0.45*\b},{0.12*\a+0.10*sin(\a*55)+0.55*\b+0.08*sin(\b*55)});}
  \foreach \b in {-0.575,0,0.575}{
    \draw[bg,smooth,samples=40,domain=-1.15:1.15,variable=\a]
      plot ({\a+0.45*\b},{0.12*\a+0.10*sin(\a*55)+0.55*\b+0.08*sin(\b*55)});}
  \foreach \a in {-1.15,1.15}{
    \draw[leaf,smooth,samples=40,domain=-1.15:1.15,variable=\b]
      plot ({\a+0.45*\b},{0.12*\a+0.10*sin(\a*55)+0.55*\b+0.08*sin(\b*55)});}
  \foreach \b in {-1.15,1.15}{
    \draw[leaf,smooth,samples=40,domain=-1.15:1.15,variable=\a]
      plot ({\a+0.45*\b},{0.12*\a+0.10*sin(\a*55)+0.55*\b+0.08*sin(\b*55)});}
  \node[below right] at ({0.6325},{-0.4766}) {$\wu_\delta(x)$};

  \foreach \a in {-0.8,0,0.8}{
    \draw[bg,smooth,samples=40,domain=-1.6:1.6,variable=\b]
      plot ({\a+0.45*\b+1},{0.12*\a+0.10*sin(\a*55)+0.55*\b+0.08*sin(\b*55)+3.05});}
  \foreach \b in {-0.8,0,0.8}{
    \draw[bg,smooth,samples=40,domain=-1.6:1.6,variable=\a]
      plot ({\a+0.45*\b+1},{0.12*\a+0.10*sin(\a*55)+0.55*\b+0.08*sin(\b*55)+3.05});}
  \foreach \a in {-1.6,1.6}{
    \draw[leaf,smooth,samples=40,domain=-1.6:1.6,variable=\b]
      plot ({\a+0.45*\b+1},{0.12*\a+0.10*sin(\a*55)+0.55*\b+0.08*sin(\b*55)+3.05});}
  \foreach \b in {-1.6,1.6}{
    \draw[leaf,smooth,samples=40,domain=-1.6:1.6,variable=\a]
      plot ({\a+0.45*\b+1},{0.12*\a+0.10*sin(\a*55)+0.55*\b+0.08*sin(\b*55)+3.05});}
  \node[above right] at ({3.32},{4.3019}) {$\wu_\epsilon(y)$};

  \foreach \p/\q in {-0.9/0.9,0.9/0.9,-0.9/-0.9}{
    \draw[bg,smooth,samples=40,domain=0:1,variable=\t]
      plot ({\p+0.45*\q+\t+0.34*sin(\t*180)},{0.12*\p+0.10*sin(\p*55)+0.55*\q+0.08*sin(\q*55)+3.05*\t});}

  \draw[ss,smooth,samples=40,domain=0:1,variable=\t]
    plot ({\t+0.7825+0.34*sin(\t*180)},{3.05*\t+0.0809});

  \node[right] at ({1.4025+0.34},{1.9719}) {$W^{ss}_\epsilon (x')$};

  \fill ({0},{0}) circle (1.7pt) node[below left] {$x$};
  \fill ({0.7825},{0.0809}) circle (1.7pt) node[below right] {$x'$};
  \fill ({1},{3.05}) circle (1.7pt) node[above left] {$y$};
  \fill ({1.7825},{3.1309}) circle (1.7pt)
        node[above right,xshift=-3pt] {$\ang{x'}{y}=h^{ss}_{xy}(x')$};

  \end{tikzpicture}
  \caption{The local $\wss$-holonomy map $h^{ss}_{xy}\colon \wu_\delta(x)\to\wu_\epsilon(y)$.}
  \label{fig:ss-holonomy}
\end{figure}

A holonomy map is H\"older if it is H\"older continuous in the metric induced from $d$ on its domain and codomain, and thus for any H\"older equivalent metric.

\subsection{Notions of H\"older continuity for smooth foliations} 
Let $\mathcal W$ be a $k$-dimensional foliation of a complete Riemannian manifold $M$.
\begin{definition}
We say that $\mathcal W$ is $\alpha$-H\"older if its tangent distribution
\[
x \longmapsto T_x\mathcal W
\]
is $\alpha$-H\"older as a section of the Grassmann bundle $\operatorname{Gr}_k(TM)$.
That is, there exists $C>0$ such that
\[
d_{\operatorname{Gr}}\bigl(T_x\mathcal W,T_y\mathcal W\bigr)
\leq C\, d(x,y)^\alpha
\]
for all sufficiently close $x,y\in M$, where $d_{\operatorname{Gr}}$ denotes a
metric on the Grassmann bundle induced by the Riemannian metric.
\end{definition}

 Following Pugh, Shub, and Wilkinson \cite{PSW}, a foliation $\mathcal{W}$ of a manifold
\emph{has H\"older holonomies} if for \emph{any} two nearby smooth transversals to $\mathcal{W}$ (with the transversals chosen to satisfy some reasonable basic control, see \cite[p.521]{PSW}, the holonomy map from one to the other along $\mathcal{W}$ is H\"older continuous with uniform constants. 

In \cite[Theorem 4]{HW}, it is shown that, if the leaves of a foliation $\WWW$ are uniformly
$C^2$
and $\WWW$ has H\"older continuous holonomies, then the tangent distribution to $\WWW$ is H\"older. While the paper \cite{HW} is focused on Anosov diffeomorphisms, their proof applies equally to foliations of pinched negative curvature Hadamard manifolds.

\subsection{Geodesic flow} Let $X$ be a proper $\CAT(-1)$ space. The space of parametrized bi-infinite geodesics in $X$,
denoted $GX$, carries a natural geodesic flow $\GGG = (g_t)_{t\in\R}$
and a standard metric $d_{GX}$, which for $c,c' \in GX$ is defined by
\[
	d_{GX}(c,c') = \int_{-\infty}^{\infty}d(c(t),c'(t))e^{-2\abs{t}}\,dt.
\]

If $X$ is a pinched negatively curved Hadamard manifold, then $(g_t): GX \to GX$
is naturally identified with the usual geodesic flow on the unit
tangent bundle which we also denote ${(g_t) : T^1 X \to T^1 X}$.
In this case, it is also natural to consider the distance $d_S$ induced by the Sasaki metric on $T^1X$, which is a natural Riemannian metric.  For a pinched negative curvature manifold, we have $d_{GX} \sim d_S$, see \S\ref{s.geodesicflow} for more details.

 For the geodesic flow over a negative curvature manifold, we let $\wguu$ denote the natural dynamical strong unstable foliation defined using the boundary at infinity and Busemann functions. This notation disambiguates this foliation from strong unstable set defined using the metric, since the two notions differ when the manifold has cusps. The subscript $\HHH$ reflects that the leaves of $\wguu$ are horospheres. By contrast, Clotet and Dal'Bo \cite{CD26} showed that a metric strong unstable set can contain many leaves of  $\wguu$, see also \cite{CTW}.

\section{Statement of Results} \label{sec:statement}

\subsection{General results} We now state the main results of our paper.

\begin{thmx}\label{prop.realstatement}
Let $\FFF: Y \times \R \to Y$ be a continuous flow on a complete separable metric space $(Y,d)$. Suppose that $\FFF$ has uniform canonical coordinates and satisfies the properties \eqref{eq.HW},
	\eqref{eq.HS}, and \eqref{eq.HF} with uniform constants $\alpha \in (0,1]$, 
	$\eta,\lambda,k > 0$, $C\geq 1$. Then the Bowen bracket of $\FFF$ is \ul{} continuous on 
	$(Y\times Y)_{\delta}$ for some $\delta > 0$. In particular, there exists $\epsilon, \delta >0$ so that for all $(x,y) \in (Y\times Y)_{\delta}$, the $u$-holonomy maps $h^u_{xy} : \wss_\delta(x) \to \wss_\epsilon(y)$ and $ss$-holonomy maps $h^{ss}_{xy} : \wu_\delta(x) \to \wu_\epsilon(y)$ are H\"older continuous with uniform H\"older exponents and constants.
\end{thmx}
The following statements are immediate corollaries.
\begin{corollary}\label{MainTheorem}
	Suppose that $\FFF$ is a $C^\alpha$ metric Anosov flow on a compact metric space $(Y,d)$.
	Then there exists $\delta >0$ so that the Bowen bracket of $\FFF$ is H\"older continuous on 
	$(Y\times Y)_{\delta}$. 
\end{corollary}

\begin{corollary} \label{Holderholonomy}
	Suppose that $\FFF$ is a $C^\alpha$ metric Anosov flow on a compact metric space $(Y,d)$.
	Then the u-holonomy maps on $\wss$ and the ss-holonomy maps on $\wu$ are H\"older continuous with uniform H\"older exponents and constants.

\end{corollary}

If $\FFF$ is a $C^\alpha$ metric Anosov flow, then so is
its time-reversed flow, which we denote $f^\iota_tx = f_{-t}x$.
The Bowen bracket $\ang{\,\cdot\,}{\,\cdot\,}^\iota$ of $f^\iota$ is given by
\[
	\{\ang{x}{y}^\iota\} = \wuu_\epsilon(x) \cap \ws_\epsilon(y).
\]
The reverse Bowen bracket can be used to define the \emph{s-holonomy maps on $\wuu$}
and \emph{uu-holonomy maps on $\ws$}. Applying Theorem \ref{prop.realstatement} to $f^\iota$, we
obtain the following complementary results.
\begin{corollary}\label{cor.realstatementi}
	Suppose that $\FFF$ satisfies the hypotheses of Proposition \ref{prop.realstatement}. Then for any
	small enough $\epsilon > 0$ there exists some $\delta > 0$ such that the Bowen bracket 
	$\ang{\,\cdot\,}{\,\cdot\,}^\iota$
	given by
	\[
		\{\ang{x}{y}^\iota\} = \wuu_\epsilon(x) \cap \ws_\epsilon(y)
	\]
	is \ul{} continuous on $(Y\times Y)_{\delta}$.
	In particular, the s-holonomies on $\wuu$ and the
	uu-holonomies on $\ws$ are H\"older continuous with uniform constants and at a uniform
	local scale.
\end{corollary}
\begin{corollary}\label{ComplementaryResult}
	Suppose that $\FFF$ is a $C^\alpha$ metric Anosov flow. Then 
	the Bowen bracket $\ang{\,\cdot\,}{\,\cdot\,}^\iota$
	is H\"older continuous on a neighborhood of the diagonal of $Y \times Y$.
	In particular, the s-holonomies on $\wuu$ and the uu-holonomies on $\ws$ are 
	H\"older continuous with uniform H\"older exponents and constants.
\end{corollary}
Our next main result investigates regularity of topological conjugacies of $C^\alpha$ metric Anosov flows.

\begin{thmx}\label{Conjugacy}
	Suppose that $\FFF$ and $\FFF'$ are $C^\alpha$ metric Anosov flows on $(Y,d)$ and $(Y',d')$,
	respectively, and suppose that $\psi : Y \to Y'$ is a topological conjugacy between $\FFF$
	and $\FFF'$. Then $\psi$ is bi-H\"older. 
\end{thmx}
\begin{remark} If $\psi : Y \to Y'$ is a topological orbit equivalence between $\FFF$
	and $\FFF'$, then we expect that for every $\epsilon > 0$ there exists a bi-H\"older orbit equivalence 
	$\psi' : Y \to Y'$ such that $d'(\psi(y),\psi'(y)) < \epsilon$ for all $y \in Y$. See the argument of \cite[Theorem 6.4.3]{FH19}.
\end{remark}
Our next main result gives a description of $C^\alpha$ metric Anosov flows in terms of symbolic dynamics. Our statement also collects some consequences of this.
\begin{thmx} \label{Coding}
	Suppose that $\FFF$ is a $C^\alpha$ metric Anosov flow. Then $\FFF$ has a strong
	Markov coding. If $\FFF$ is transitive, then so is the strong Markov coding. When $\FFF$ is transitive, it thus follows that an equilibrium state for a H\"older
	potential is unique and satisfies the Almost Sure Invariance Principle,  the Law of the Iterated
	Logarithm, and the Central Limit Theorem. If the lengths of periodic orbits are not contained in
	$c \Z$, then it has the Bernoulli property.
\end{thmx}

\subsection{Results for geodesic flows}
We apply our results to geodesic flows in negative curvature. The following is an immediate corollary of Theorem \ref{prop.realstatement} applied to geodesic flows in negative curvature (on the universal cover).
\begin{corollary}\label{GeodesicFlow}
	Suppose that $\FFF : Y\times \R \to Y$ is either
	\begin{enumerate}
		\item the geodesic flow $\GGG : GX \times \R	\to GX$ 
			of any proper $\CAT(-1)$ space $X$ with respect to $d = d_{GX}$, or
		\item the geodesic flow $\GGG : T^1X \times \R \to T^1X$ 
			of any pinched negatively curved Hadamard manifold $X$ with respect to $d = d_S$.
	\end{enumerate}
	Then the Bowen bracket of $\FFF$ is \ul{} continuous on $(Y\times Y)_\delta$ for some $\delta > 0$. The u-holonomy maps on $\wss_\delta$, s-holonomy maps on $\wuu_\delta$, uu-holonomy maps
	on $\ws_\delta$, and ss-holonomy maps on $\wu_\delta$ are H\"older with uniform constants and exponent.
\end{corollary}
\begin{remark} \label{rem:limitiations}
We emphasize that the above result is formulated so that it applies for a geodesic flow `upstairs' in the universal cover. If we consider the geodesic flow for $X_0 = X / \Gamma$ where $\Gamma$ is a non-elementary torsion-free Kleinian group, the space $X_0$ is locally $\CAT(-1)$ rather than globally $\CAT(-1)$ so the above result does not apply directly to $GX_0$. When $X_0$ is compact, the results descend verbatim. When $X_0$ is non-compact, one has to take care with how all the objects in this result are defined.  In the non-compact case, for sufficiently small scales depending on the local injectivity radius, the above results descend to $X_0$ as long as we consider the projections of all the objects defined upstairs. However, these results do not descend to $X_0$ at a uniform scale, and the Bowen bracket for the geodesic flow for $X_0$ is not well-defined even at small scales in the presence of a cusp.
\end{remark} 

While it is natural to expect H\"older continuity of the Bowen bracket, issues of H\"older regularity of the stable and unstable are subtle even in the manifold case. Of particular interest to us for regularity questions are the family of smooth surfaces constructed by Ballmann, Brin and Burns in \cite{BBB}. We call these the $B^3$-examples to reflect the names of the authors. The $B^3$-examples are smooth surfaces with pinched negative curvature whose strong stable foliations are not H\"older in the distributional sense.   This contrasts with the theorem that manifolds with pinched negative curvature for which the first derivatives of curvature are uniformly bounded have H\"older continuous foliations in the distributional sense, see \cite[Theorem 7.3]{PPS}.  Even though the $B^3$ examples have non-H\"older distributions, we use Corollary \ref{GeodesicFlow} to prove the following result\footnote{We emphasize that while this result can likely also be proved by more direct geometric arguments, it follows readily from our framework, and we have not found it stated in the literature. For instance, the maps $\theta_x^{-1}\circ\theta_{x'}$ in \cite[\S2.2]{PP16},  which are holonomy maps of the weak stable foliation, are stated to be H\"older under the additional assumption that the derivatives of the curvature are bounded.}. Recall that $\wguu$ is the smooth dynamical strong stable foliation which can be defined using horospheres.
\begin{thmx} \label{Holonomies}
	Let $(M,\mathfrak{g})$ be a complete connected pinched negatively curved manifold. Then $\wguu$, $\wgu$, $\wgss$, and $\wgs$ have H\"older holonomies.
\end{thmx}
This adds interest to our result since it is a standard result that H\"older continuity of distributions implies H\"older continuity of holonomy maps under a `mild' additional assumption, which is that the distributions are uniformly $C^2$. The reason that there is no contradiction is that this `mild' additional assumption must fail for the $B^3$-examples. We obtain the following corollary. 

\begin{corollary} \label{cor:B3}
	The foliation $\wgss$ for a $B^3$-example is not uniformly $C^2$.
\end{corollary}

\begin{remark}
H\"older continuity of the foliations in the distributional sense is a
main ingredient in obtaining absolute continuity, which is used in the
Hopf argument for ergodicity of the Liouville measure.
For pinched negatively curved manifolds without bounds on the first derivative of the curvatures, the ergodicity of Liouville measure thus remains open. This explains the relevance of the $B^3$-examples.  Theorem \ref{Holonomies} does not seem to be helpful for ergodicity but it does clarify what one can still get without the assumption of bounded derivatives of curvature.
\end{remark}

\section{H\"older continuity of local product structure}\label{s.main}
We prove Theorem \ref{prop.realstatement}. Suppose that $\FFF : Y \times \R \to Y$ is a continuous flow on a complete separable metric space $(Y,d)$.
We do not assume in this section that $Y$ is compact. We assume that $\FFF$ satisfies the conditions of Proposition
\ref{prop.realstatement}, and we fix the uniform constants $\alpha \in (0,1]$, $\eta,\lambda,k > 0$,
$C \geq 1$ from \eqref{eq.HW}, \eqref{eq.HS}, and \eqref{eq.HF}. Proposition
\ref{prop.realstatement} is proved in two parts, which are Propositions \ref{lem.2ndvar} and \ref{lem.1stvar}.

Given $0 < \delta \leq \eta$, we set
\begin{equation}\label{def:L}
	L(\delta) := \max\Big\{\tfrac1k \log \tfrac{C}{\delta},\, 0\Big\} \geq 0,
\end{equation}
and for $x,y \in Y$ with $x \neq y$ we define
\begin{equation}\label{def:T}
	T(x,y;\delta) := \max\Big\{-\tfrac{\alpha}{k}\log d(x,y) - L(\delta),\, 0\Big\}.
\end{equation}
A short calculation from the definitions gives that for any $x,y \in Y$ with $x \neq y$, and $\lambda>0$, we have
\begin{equation}\label{relateTandL}
e^{-\lambda T(x, y;\delta) } \leq e^{\lambda L(\delta)} d(x, y) ^{\lambda \alpha  / k}
\end{equation}
with equality unless $T(x, y;\delta)=0$. The following elementary lemma, which uses only \eqref{eq.HF},  says that if $x$ and $y$ are $\delta$-close, then
the points $f_t x$ and $f_t y$ stay $\delta$-close for all times $\abs{t} \leq T(x,y;\delta)$, where as $y \to x$ the quantity
$T(x,y;\delta)$ grows to infinity logarithmically. 

\begin{lemma}\label{lem.logclose}
	Let $0 < \delta \leq \eta$. If $x,y \in Y$ satisfy $x \neq y$ and
	$d(x,y) \leq \delta$, and $t \in \R$ satisfies $\abs{t} \leq T(x,y;\delta)$, then
	\[
		d(f_tx,f_ty) \leq \delta . 
	\]
\end{lemma}

\begin{proof}
	If $T(x,y;\delta) = 0$ then $t = 0$ and $d(f_0x,f_0y) = d(x,y) \leq \delta$, as
	required. Suppose then that $T(x,y;\delta) > 0$, so that by \eqref{def:T},
	\[
		\abs{t} \leq -\tfrac{\alpha}{k}\log d(x,y) - L(\delta).
	\]
	Multiplying by $k$ and exponentiating, we obtain from \eqref{def:L} that
	\[
		e^{k\abs{t}} \leq d(x,y)^{-\alpha} e^{-kL(\delta)}
		\leq \frac{\delta}{C}\, d(x,y)^{-\alpha}.
	\]
	 Since $d(x,y) \leq \delta \leq \eta$, the hypothesis
	\eqref{eq.HF} applies and yields
	\[
		d(f_tx,f_ty) \leq C e^{k\abs{t}} d(x,y)^{\alpha}
		\leq C \cdot \frac{\delta}{C}\, d(x,y)^{-\alpha} \cdot d(x,y)^{\alpha} = \delta . \qedhere
	\]
\end{proof}

\begin{remark}\label{rem:truncations}
	 If we additionally assume that the constant $\delta$ satisfies $\delta \leq 1$, we can simplify the bookkeeping above. We have
	$L(\delta) = \frac1k\log\frac{C}{\delta}$. If
	$d(x,y) < (\delta/C)^{1/\alpha}$, then $d(x,y) \leq \delta$ and $T(x,y;\delta) = -\tfrac{\alpha}{k}\log d(x,y) - L(\delta)> 0$. Lemma~\ref{lem.logclose} can then be stated with the hypothesis
	$0 < d(x,y) < (\delta/C)^{1/\alpha}$, which is exactly the hypothesis needed for the  conclusion to be non-trivial.
\end{remark}

The following lemma shows that for a fixed $x$, the map $ y \to \ang{x}{y}$ is uniformly locally H\"older. In the case that $Y$ is compact, it follows that $y \to \ang{x}{y}$ is H\"older. Recall  that varying the second argument in $\ang{x}{y}$ corresponds to staying on the stable set of $x$ and moving across the weak unstable sets.
\begin{proposition}\label{lem.2ndvar}
	Let $\delta>0$ be sufficiently small. There exist $K_1= K_1(\delta) > 0$ such that if $x,y_1,y_2 \in Y$ satisfy $d(x,y_1) \leq \delta$ and $d(x,y_2) \leq \delta$, then
	\[
	d(\ang{x}{y_1},\ang{x}{y_2}) \leq K_1d(y_1,y_2)^{\frac{\lambda\alpha}{k}}
	.\]
\end{proposition}
\begin{proof}
	Since $\FFF$ has uniform canonical coordinates, then for any sufficiently small $\epsilon>0$, we can find $\delta_0>0$ so that for any $ \delta \in (0, \delta_0)$, the flow has uniform canonical coordinates at scale $(\epsilon, \delta)$. We assume without loss of generality that $\epsilon > 0$ and $\delta >0$ are chosen so that 
	\begin{equation} \label{eq:epsdeltaconvention}
	2\epsilon + 2C\epsilon^\alpha \leq \frac{1}{2}\eta \hspace{30pt} \delta \leq \frac{1}{4}\eta.
	\end{equation}

	We fix $x, y_1, y_2 \in Y$ with $d(x,y_i) \leq \delta$. Note that $d(y_1,y_2) \leq
	2\delta$.  Let $L=L(2\delta)$ and $T= T(y_1,y_2, 2\delta) \geq 0$ be as in \eqref{def:L} and \eqref{def:T} respectively.
	By Lemma \ref{lem.logclose}, 
	\begin{equation} \label{eq.logclose2delta}
	d(f_ty_1,f_ty_2) \leq 2\delta \mbox{ for all } \abs{t} \leq T.
	\end{equation}

	Let $z_1 = \ang{x}{y_1}$ and $z_2 = \ang{x}{y_2}$. 
	For $i = 1, 2$, there exists $v_i = v(x,y_i)$ with $\abs{v_i} \leq \epsilon$
	such that $z_i \in \wuu_\epsilon(f_{v_i}y_i)$. For $T= T(y_1,y_2, 2\delta)$, using the fact that $z_i \in \wuu_\epsilon(f_{v_i}y_i)$ and the properties \eqref{eq.logclose2delta} and \eqref{eq.HS}, we calculate that
\[
		\begin{split}
			d(f_{-T}z_1,f_{-T}z_2) &\leq d(f_{-T}z_1,f_{-T+v_1}y_1)+d(f_{-T+v_1}y_1,f_{-T}y_1)+d(f_{-T}y_1,f_{-T}y_2)\\
											 & \quad +d(f_{-T}y_2,f_{-T+v_2}y_2) +d(f_{-T+v_2}y_2,f_{-T}z_2)\\
											 &\leq \epsilon + C\epsilon^\alpha + 2\delta + C\epsilon^\alpha + \epsilon
											 \leq \eta.
		\end{split}
		\]
	Since $z_1$ and $z_2$ are in the same strong stable set, it follows that $f_{-T}z_1 \in \wss_{\eta}(f_{-T}z_2)$. It follows from \eqref{eq.HW} that
	\[
	d(z_1,z_2) \leq Ce^{-\lambda T} d(f_{-T}z_1, f_{-T}z_2)^\alpha \leq Ce^{-\lambda T}\eta^\alpha.
	\]
Thus, using \eqref{relateTandL}, we have
	\[
		d(\ang{x}{y_1},\ang{x}{y_2}) = d(z_1,z_2) \leq Ce^{-\lambda T}\eta^\alpha
		\leq C\eta^\alpha e^{\lambda L}d(y_1,y_2)^{\frac{\lambda\alpha}{k}}.
	\] 
	The result follows with $K_1 = C\eta^\alpha e^{\lambda L(\delta)}$.
\end{proof}
The following lemma shows that for a fixed $y$, the map $ x \to \ang{x}{y}$ is uniformly locally H\"older. In the case that $Y$ is compact, it follows that $x \to \ang{x}{y}$ is H\"older. Recall  that varying the first argument in $\ang{x}{y}$ corresponds to staying on the weak unstable of $y$ and moving across the strong stable sets.
\begin{proposition}\label{lem.1stvar}
	Let $\delta>0$ be sufficiently small. There exist $K_2=K_2(\delta) > 0$ such that, if $x_1,x_2,y \in Y$ satisfy  $d(x_1,y) \leq \delta$ and $d(x_2,y) \leq \delta$, then
	\[
	d(\ang{x_1}{y},\ang{x_2}{y}) \leq K_2 d(x_1,x_2)^{\frac{\lambda\alpha^3}{2k}}
	.\]
\end{proposition}

\begin{proof} We choose $\epsilon>0$ and $\delta>0$ the same way as in the previous proof. In particular, the flow has uniform canonical coordinates at scale $(\epsilon, \delta)$ and \eqref{eq:epsdeltaconvention} holds.

	We fix $x_1, x_2, y \in Y$ with $d(x_i,y) \leq \delta$. Note that $d(x_1,x_2) \leq
	2\delta$.  Let $L=L(2\delta)$ and $T= T(x_1,x_2, 2\delta) \geq 0$ be as in \eqref{def:L} and \eqref{def:T} respectively.
	By Lemma \ref{lem.logclose}, 
	\begin{equation} \label{eq.xlogclose2delta}
	d(f_tx_1,f_tx_2) \leq 2\delta \mbox{ for all } \abs{t} \leq T.
	\end{equation}

	Let $w_1 = \ang{x_1}{y}$ and $w_2 = \ang{x_2}{y}$. 
	For $i = 1, 2$, there exists $v_i = v(x_i,y)$ with $\abs{v_i} \leq \epsilon$
	such that $w_i \in \wuu_\epsilon(f_{v_i}y)$. By possibly switching $w_1$ with $w_2$, we assume that $v = v(x_1, x_2) \coloneqq v_1-v_2 \geq 0$. 
	Let $w_1:'= f_{-v}w_1$. Since $w_1 \in \wuu_\epsilon(f_{v_1}y)$, it is easy to check that $w_1' \in \wuu_{2\epsilon} (w_2) \subset \wuu_{\eta} (w_2)$.  

For $T= T(x_1,x_2, 2\delta)$, using \eqref{eq.HS}, the fact that $w_i \in \wss_\epsilon(x_i)$ and the property \eqref{eq.xlogclose2delta}, we calculate that

	\begin{equation}\label{eq.1stvar+}
		\begin{split}
			d(f_{T}w'_1,f_T w_2) &\leq d(f_{-v}(f_{T}w_1),f_Tw_1) + d(f_Tw_1,f_T x_1) \\
													 &\quad + d(f_Tx_1,f_Tx_2)+d(f_Tx_2,f_Tw_2) \\
													 &\leq C(2\epsilon)^\alpha + \epsilon + 2\delta + \epsilon \leq \eta.
		\end{split}
	\end{equation}
	Since $w'_1$ and $w_2$ are in the same strong unstable set, we have
	$f_{T}w'_1 \in \wuu_{\eta}(f_Tw_2)$.   It follows from \eqref{eq.HW} that
	\[
	d(w'_1,w_2) \leq Ce^{-\lambda T} d(f_{T}w'_1, f_{T}w_2)^\alpha \leq Ce^{-\lambda T}\eta^\alpha, 
	\]
and thus using \eqref{relateTandL}, we have
	\[
			d(w'_1,w_2) \leq C\eta^\alpha e^{\lambda L}d(y_1,y_2)^{\frac{\lambda\alpha}{k}}.
	\] 
It follows that
	\begin{equation}\label{eq.1stvarfinal}
		\begin{split}
			d(\ang{x_1}{y}),\ang{x_2}{y}) = d(w_1,w_2) &\leq d(w_1,f_{-v}w_1) + d(w'_1,w_2) \\
								& \leq Cv^{\alpha}+C\eta^\alpha e^{\lambda L}d(x_1,x_2)^{\frac{\lambda\alpha}{k}}.
		\end{split}
	\end{equation}
	It remains to obtain a bound on $v= v(x_1, x_2)$.  If $v=0$, we are done, so we assume that $v>0$. Using \eqref{eq.HS}, we have
	\begin{align*}
		\frac{1}{C}v^{\frac{1}{\alpha}} &\leq d(f_{\frac{T}{2}-v}w_1,f_{\frac{T}{2}}w_1) \\
									 &\leq d(f_{\frac{T}{2}-v}w_1,f_{\frac{T}{2}}w_2) + d(f_{\frac{T}{2}}w_2,f_{\frac{T}{2}}x_2)+ d(f_{\frac{T}{2}}x_2, f_{\frac{T}{2}} x_1) + d(f_{\frac{T}{2}}x_1, f_{\frac{T}{2}}w_1).									
	\end{align*}
We address each term in the above estimate.  By \eqref{eq.1stvar+}, $f_Tw_1' \in  \wuu_{\eta} (f_Tw_2)$, so by \eqref{eq.HW} we have
\[
d(f_{\frac{T}{2}-v}w_1,f_{\frac{T}{2}}w_2) 	= d(f_{\frac{T}{2}}w'_1,f_{\frac{T}{2}}w_2) \leq C d(f_{{T}}w'_1,f_{{T}}w_2)e^{-\frac{\lambda T}{2}} \leq C \eta e^{-\frac{\lambda T}{2}}.
\]					 
For the second term, and the same argument for the fourth term, since $w_2 \in \wss_\epsilon (x_2)$ by \eqref{eq.HW} we have		
	\[
	d(f_{\frac{T}{2}}w_2,f_{\frac{T}{2}}x_2)  \leq C d(w_2,x_2)e^{-\frac{\lambda T}{2}} \leq C \eta e^{-\frac{\lambda T}{2}}.
	\]
	Using \eqref{eq.HF} on the third term and combining the above, we obtain
	\[
	\frac{1}{C}v^{\frac{1}{\alpha}} \leq 3 Ce^{-\frac{\lambda T}{2}} + Ce^{\frac{kT}{2}}d(x_1,x_2)^\alpha.
	\]

By \eqref{relateTandL}, we have
\[
	e^{-\frac{\lambda T}{2}} \leq e^{\frac{\lambda L}{2}}\, d(x_1,x_2)^{\frac{\lambda\alpha}{2k}}.
\]

To bound the second term, we assume  that $d(x_1,x_2) < e^{-\frac{k L(\delta)}{\alpha}}$. This ensures that
$T(x_1,x_2;\delta) > 0$ and that  $d(x_1,x_2) \leq 1$. The definition of $T(x_1,x_2;\delta)$ gives
\[
	e^{\frac{kT}{2}} \leq e^{-\frac{kL}{2}}\, d(x_1,x_2)^{-\frac{\alpha}{2}}.
\]
Using $d(x_1,x_2)^{\alpha}\leq d(x_1,x_2)^{\frac{\alpha}{2}}$, we obtain
\[
	\frac{1}{C}v^{\frac{1}{\alpha}}
	\leq 3Ce^{\frac{\lambda L}{2}}\, d(x_1,x_2)^{\frac{\lambda\alpha}{2k}}
	   + Ce^{-\frac{kL}{2}}\, d(x_1,x_2)^{\frac{\alpha}{2}}.
\]
Since $\lambda \leq k$ we have $\frac{\lambda\alpha}{2k}\leq\frac{\alpha}{2}$, and hence
$d(x_1,x_2)^{\frac{\alpha}{2}} \leq d(x_1,x_2)^{\frac{\lambda\alpha}{2k}}$. Therefore
\[
	\frac{1}{C}v^{\frac{1}{\alpha}} \leq C_1 d(x_1,x_2)^{\frac{\lambda\alpha}{2k}}
	\]
	where $C_1 = C_1(\delta) := 3Ce^{\frac{\lambda L(\delta)}{2}} + Ce^{-\frac{kL(\delta)}{2}}$. We obtain
\begin{equation} \label{eq:estonv}
	v \leq (C_1C_2)^\alpha\, d(x_1,x_2)^{\frac{\lambda\alpha^{2}}{2k}}.
\end{equation}
By increasing the constant if necessary to remove the assumption that $d(x_1,x_2) < e^{-\frac{k L(\delta)}{\alpha}}$, we obtain a constant $C_3$ such that
\begin{equation} \label{eq:estonv2}
	v \leq C_3\, d(x_1,x_2)^{\frac{\lambda\alpha^{2}}{2k}}.
\end{equation}

It follows from \eqref{eq.1stvarfinal} that
	\begin{equation}
		\begin{split}
			d(\ang{x_1}{y}),\ang{x_2}{y}) & \leq Cv^{\alpha}+C\eta^\alpha e^{\lambda L}d(x_1,x_2)^{\frac{\lambda\alpha}{k}}. \\
								& \leq CC_3^\alpha d(x_1,x_2)^{\frac{\lambda\alpha^{3}}{2k}}+C\eta^\alpha e^{\lambda L}d(x_1,x_2)^{\frac{\lambda\alpha}{k}},
		\end{split}
	\end{equation}
and the result follows. 
	\end{proof}
We note that we could have used $p T$ for a fixed $p \in (0,1)$ in place of $\frac{T}{2}$ in the part of the argument above which starts at \eqref{eq:estonv}. We used $p = \frac{1}{2}$ since it gives the best exponent in \eqref{eq:estonv2}.
\begin{remark}\label{rem.ouroptimal}
In the Lipschitz case when $\alpha=1$, the exponents are $\frac{\lambda}{2k}$ in the first variable and $\frac{\lambda}{k}$ in the second variable, where $\lambda \leq k$ are the constants from \eqref{eq.HW} and \eqref{eq.HF}. The best possible exponents provided by our argument thus correspond to the case that $\lambda=k=1$, and give exponents of $\frac{1}{2}$ in the first variable and $1$ in the second variable.  In particular, the best exponents possibly provided by our proof are $\frac{1}{2}$ for the ss-holonomies on $\wu$ and $1$ for the u-holonomies on $\wss$.
\end{remark}
\begin{remark}\label{rem.wilkinson}
	If $f$ is a $C^1$ Anosov flow, then the proof of \cite[Corollary 5.3]{Wilkinson} shows
	that the Bowen bracket is H\"older in its first variable with exponent
	$\frac{\lambda}{\lambda+k}$ using our notation.
	In the $C^1$ case, this is better than our exponent of $\frac{\lambda}{2k}$ 
	from Lemma \ref{lem.1stvar}, although both results are still limited to exponents at most 
	$\frac{1}{2}$.
	The proof in \cite{Wilkinson} does not extend to the $C^\alpha$ or Lipschitz setting of the
	current paper, because it would require a growth estimate on weak unstable sets which would take
	the place of  \cite[Lemma 5.4 (2)]{Wilkinson}. Our definition of $C^\alpha$ metric Anosov gives control on the neutral direction and the strong unstable direction separately, but our setup does not a priori yield any good estimates on the weak unstable sets.
\end{remark}

\begin{remark}\label{rem.optimal}
In the case of $C^{1+\alpha}$ Anosov flows, there are various strategies to obtain improved (and often optimal) exponents of the stable and unstable holonomies in \cite{PSW, SS, HW}. 
	Exponents  can be improved by distinguishing between the forward and backward expansion and contraction rates and by considering their pointwise behavior. 
	These arguments could be investigated in our setting. The maximum exponents for the Bowen bracket that our essential proof strategy provides is still $\frac{1}{2}$ in the first variable and $1$ in the second variable.
\end{remark}

We conclude the section by showing that the analogous result to Theorem \ref{prop.realstatement}
holds for the Bowen bracket $\ang{\,\cdot\,}{\,\cdot\,}^\iota$ of the reverse flow $f^\iota$.

\begin{lemma}\label{lem.reverse}
	Suppose that $f$ satisfies the hypotheses of Theorem \ref{prop.realstatement}. Then so does
	the reverse flow $f^\iota$.
\end{lemma}
\begin{proof}
	The properties \eqref{eq.HW}, \eqref{eq.HS}, and \eqref{eq.HF} are all clearly reversible,
	so we just have to check that $f^\iota$ has uniform canonical coordinates.
	Choose an initial scale $(\epsilon,\delta)$ for the uniform canonical coordinates of $f$,
	and suppose that $0 < \epsilon' \leq \epsilon$. Suppose that $x,y \in Y$, $\abs{v} \leq
	\epsilon'$, and $\abs{v'} \leq \epsilon'$ are all given such that $d(x,y) \leq \delta$ and
	\[
		z \in \wuu_{\epsilon'}(x) \cap \wss_{\epsilon'}(f_vy), \quad z' \in \wuu_{\epsilon'}(x) \cap
		\wss_{\epsilon'}(f_{v'}y)
	.\]
	Assuming we had chosen $\epsilon'$ sufficiently small,  using uniform continuity of $f$ we can
	ensure that
	\[
		f_{-v}z \in \wss_\epsilon(y) \cap \wuu_\epsilon(f_{-v}x), \quad f_{-v'}z' \in \wss_\epsilon(y)
		\cap \wuu_\epsilon(f_{-v'}x)
	,\]
	and hence $v=v'$ and $z = z'$ because $f$ has uniform canonical coordinates.
	This shows uniqueness of $\ang{x}{y}^\iota$ at the scale $(\epsilon',\delta)$.

	Now let $(\epsilon'',\delta')$ be another scale for the canonical coordinates of
	$f$ such that $0 < \epsilon'' \leq \epsilon'$ and $0 < \delta' \leq \delta$.
	If $d(x,y) \leq \delta'$, then we have
	\[
	\ang{y}{x} \in \wss_{\epsilon''}(y) \cap \wuu_{\epsilon''}(f_v x), \quad \abs{v} \leq \epsilon''
	,\]
	and again using uniform continuity we can assume that $\epsilon''$ was chosen small enough
	that
	\[f_{-v}\ang{y}{x} \in \wuu_{\epsilon'}(x) \cap \wss_{\epsilon'}(f_{-v}y).\]
	This shows existence of $\ang{x}{y}^\iota = f_{-v}\ang{y}{x}$, and thus 
	$(\epsilon',\delta')$ is a scale for both the existence and uniqueness of
	uniform canonical coordinates for $f^\iota$.
\end{proof}

\section{Topological conjugacy of H\"older flows}\label{subsec.conjugacy}
We develop some more detailed results on the local metric structure of $C^\alpha$ metric Anosov flows and
prove Theorem \ref{Conjugacy}. In this section, the phase space is assumed to be compact to obtain the most natural setting and statement for Theorem \ref{Conjugacy}.
In Proposition \ref{rem.noncompactB}, we state the analogue of Theorem \ref{Conjugacy} that holds when compactness is relaxed. Following Ruelle's presentation of \cite[\S7.1]{Ruelle}, we begin with the following general definition which can be used to establish an abstract notion of local product structure at a topological level. See also Putnam's work on Smale spaces \cite{iP14}.
\begin{definition}\label{def.lps}
	Let $Y$ be a compact space.
	An \emph{(abstract) local bracket operation} on $Y$ is a continuous function
	\[ 
		[\,\cdot\, , \,\cdot\,] : (Y \times Y)_\delta \to Y
	\]
	for some $\delta>0$ which satisfies
	\begin{enumerate}[label=(\alph*)]
		\item $[x,x] =x$,\label{item.a}
		\item $[[x,y],z]=[x,z]$,\label{item.b}
		\item $[x,[y,z]]=[x,z]$,\label{item.c}
	\end{enumerate}
	whenever both sides of the equations are defined.
\end{definition}

Given an abstract local bracket $[\,\cdot\, , \,\cdot\,]$ on a compact metric space $(Y,d)$,
we define
\begin{gather*}
	L_\epsilon(y) = \{ x \mid x = [x,y]\mathrm{~and~ }  d(y, x)< \epsilon \}\\
	R_\epsilon(y) = \{ x \mid x = [y,x] \mathrm{~and~} d(y, x)< \epsilon \}.
\end{gather*}
for each $y \in Y$, so that $L_\epsilon(y)$ is the set of points fixed by the left side of the bracket operation with $y$ and $R_\epsilon(y)$ is the set of points fixed by the right side of the bracket operation with $y$. We then have the following lemmas.

\begin{lemma}\label{lem.lpscc}
	For every small enough $\epsilon > 0$, there is some $\delta > 0$ such that, if $d(x,y) \leq \delta$,
	then $\{[x,y]\} = R_\epsilon(x) \cap L_\epsilon(y)$.
\end{lemma}
\begin{proof}
	Using uniform continuity of $[\,\cdot\, ,\,\cdot\,]$,
	choose $\delta' > 0$ so that,
	for any $y_1,y_2,y_3 \in Y$ with maximum distance at most
	$\delta'$, then $[y_1,[y_2,y_3]]$ and $[[y_1,y_2],y_3]$ are well-defined.
	Let $0 < \epsilon \leq \delta'$ be given, and choose $0 < \delta \leq \delta'$ such that,
	for any $x,y \in Y$, then $d(x,y) \leq \delta$ implies that $d(x,[x,y]) = < \epsilon$
	and $d(y,[x,y]) < \epsilon$.

	Fix $x,y \in Y$ such that $d(x,y) \leq \delta$.
	First, suppose $z \in R_\epsilon(x) \cap L_\epsilon(y)$. Note that $x,y,z$
	have maximum distance at most $\delta'$, and in particular $[x,[z,y]]$, $[x,y]$, $[[x,y],y]$, and
	$[x,[x,y]]$ are all well-defined.
	We know that $z = [x,z] = [z,y]$ by definition, and hence 
	\[z = [z,z] = [[x,z],[z,y]] = [x,[z,y]] = [x,y].\]
	On the other hand, since $[x,y] = [[x,y],y] = [x,[x,y]]$, and we have $d(x,[x,y]) < \epsilon$
	and $d(y,[x,y]) < \epsilon$, then $[x,y] \in R_\epsilon(x) \cap
	L_\epsilon(y)$.
\end{proof}

\begin{lemma}[{\cite[\S7.1]{Ruelle}}]\label{lem.ruelle}
	For any small enough $\delta>0$ the map 
	\[
		[\,\cdot\, , \,\cdot\,] : L_\delta(y)  \times R_\delta(y) \to Y
	\]
	is a homeomorphism onto an open neighborhood of $y$ for any $y \in Y$, and
	the inverse of this map is given by
	\[
		x \mapsto ([x,y],[y,x])
	.\]
\end{lemma}

\begin{lemma}\label{lem.lps}
	Suppose that $f$ is a $C^\alpha$ metric Anosov flow on $Y$. Then the
	Bowen brackets $\ang{\,\cdot\,}{\,\cdot\,}$ and
	$\ang{\,\cdot\,}{\,\cdot\,}^\iota$ of $f$ are abstract local bracket operations on some neighborhood of the diagonal in $Y \times Y$.
\end{lemma}
\begin{proof}
The Bowen brackets are continuous on a neighborhood of the diagonal
by Proposition \ref{prop.realstatement} and Corollary \ref{cor.realstatementi}.
It is then easy to check from their definitions that they satisfy \ref{item.a},
\ref{item.b}, and \ref{item.c} on a possibly smaller neighborhood of the diagonal, and hence they
are local bracket operations on this neighborhood.
\end{proof}

	Since $\ang{\,\cdot\,}{\,\cdot\,}$ and $\ang{\,\cdot\,}{\,\cdot\,}^\iota$ are bracket operations by Lemma \ref{lem.lps}, and since they are H\"older by Corollaries \ref{MainTheorem} and \ref{ComplementaryResult}, then we obtain the following corollary directly
	from Lemma \ref{lem.ruelle} and the definitions of the Bowen brackets.
\begin{corollary}\label{cor.lps}
	Suppose that $f$ is a $C^\alpha$ metric Anosov flow on $Y$. Then 
	for any small enough $\epsilon > 0$, there is some $\delta > 0$ such that for any $y \in Y$, the maps
	\begin{gather*}
		\ang{\,\cdot\,}{\,\cdot\,} : (\wu_\epsilon(y) \cap B(y,\delta)) \times (\wss_\epsilon(y) \cap
		B(y,\delta)) \to Y,\\
		\ang{\,\cdot\,}{\,\cdot\,}^\iota : (\ws_\epsilon(y) \cap B(y,\delta)) \times (\wuu_\epsilon(y) \cap
		B(y,\delta)) \to Y
	\end{gather*}
	are bi-H\"older homeomorphisms onto open neighborhoods of $y$.
\end{corollary}

It is convenient to work with only one scale $\tau$ instead of the two
scales $\epsilon$ and $\delta$. The following lemma allows us to make this adjustment.
\begin{lemma}\label{lem.subset}
	Suppose that $f$ is a $C^\alpha$ metric Anosov flow on $Y$.
	For any $\epsilon,\delta > 0$ there is some $\tau > 0$ such that, for every $y \in Y$, we have
	\begin{gather*}
		\wuu_\tau(y) \subseteq \wuu_\epsilon(y)\cap B(y,\delta), \quad
		\wss_\tau(y) \subseteq \wss_\epsilon(y)\cap B(y,\delta), \\
		\ws_\tau(y) \subseteq \ws_\epsilon(y)\cap B(y,\delta), \quad
		\wu_\tau(y) \subseteq \wu_\epsilon(y)\cap B(y,\delta).
	\end{gather*}
	Furthermore, for any small enough $\epsilon > 0$, for any $\tau > 0$ there is some $\delta > 0$
	such that, for every $y \in Y$, we have
	\begin{gather*}
		\wuu_\epsilon(y)\cap B(y,\delta) \subseteq \wuu_\tau(y), \quad
		\wss_\epsilon(y)\cap B(y,\delta) \subseteq \wss_\tau(y), \\
		\ws_\epsilon(y)\cap B(y,\delta) \subseteq \ws_\tau(y), \quad
		\wu_\epsilon(y)\cap B(y,\delta) \subseteq \wu_\tau(y).
	\end{gather*}
\end{lemma}
\begin{proof}
	The first group of inclusions is obvious by definitions and the uniform continuity of the flow.

	For the second group of inclusions, we use the canonical coordinates of $f$ and $f^\iota$. Choose some $\epsilon_0 > 0$ such that, for any $x,y,z \in Y$, then $x \in
	\wss_{\epsilon_0}(y)\cap\wu_{\epsilon_0}(z)$ implies that $x = \ang{y}{z}$, and $x \in
	\wuu_{\epsilon_0}(y)\cap\ws_{\epsilon_0}(z)$ implies that $x = \ang{y}{z}^\iota$.
	Let $0 < \tau,\epsilon \leq \epsilon_0$ be chosen small enough so that there is some $\delta > 0$
	such that $(\tau,\delta)$ is a scale for the canonical coordinates of $f$ and $f^\iota$,
	and for any $x,y,z \in Y$, then
	\[
		z \in W^\ell_\tau(y) \text{ and } x \in W^\ell_\epsilon(y) \implies z \in W^\ell_{\epsilon_0}(x)
	\]
	for any $\ell \in \{\mathrm{ss},\mathrm{uu},\mathrm{s},\mathrm{u}\}$.
	Suppose that $x \in \wss_\epsilon(y)\cap B(y,\delta)$. Then $d(x,y) < \delta$, so
	\[
		\ang{y}{x} \in \wss_\tau(y) \cap \wu_\tau(x) \subseteq
		\wss_{\epsilon_0}(x)\cap\wu_{\epsilon_0}(x)
	,\]
	and therefore $x = \ang{x}{x} = \ang{y}{x} \in \wss_\tau(y)$.
	Similarly, if $x \in \wu_\epsilon(y)\cap B(y,\delta)$, then
	\[
		\ang{x}{y} \in \wss_\tau(x)\cap \wu_\tau(y) \subseteq \wss_{\epsilon_0}(x) \cap
		\wu_{\epsilon_0}(x)
	,\]
	and therefore $x = \ang{x}{x} = \ang{x}{y} \in \wu_\tau(y)$. Thus we have shown that, if $\epsilon
	> 0$ and $\tau > 0$ are sufficiently small, then there is some $\delta > 0$ such that
	\[
		\wss_\epsilon(y)\cap B(y,\delta) \subseteq \wss_\tau(y), \quad
		\wu_\epsilon(y)\cap B(y,\delta) \subseteq \wu_\tau(y)
	.\]
	Clearly this implies the same for all $\tau > 0$. Using the reverse Bowen bracket
	$\ang{\,\cdot\,}{\,\cdot\,}^\iota$, then by the same proof we obtain the inclusions
	\[
		\wuu_\epsilon(y)\cap B(y,\delta) \subseteq \wuu_\tau(y), \quad
		\ws_\epsilon(y)\cap B(y,\delta) \subseteq \ws_\tau(y). \qedhere
	\]
	\end{proof}

We have the following immediate consequence of Corollary \ref{cor.lps} and Lemma
\ref{lem.subset}.

\begin{corollary}\label{cor.lpstau}
Suppose that $f$ is a $C^\alpha$ metric Anosov flow on $Y$.
Then for any sufficiently small $\tau > 0$, the maps
\[
\ang{\,\cdot\,}{\,\cdot\,} : \wu_\tau(y) \times \wss_\tau(y) \to Y, \quad
\ang{\,\cdot\,}{\,\cdot\,}^\iota : \ws_\tau(y) \times \wuu_\tau(y) \to Y
\]
are bi-H\"older homeomorphisms onto neighborhoods of $y$ for any $y \in Y$.
\end{corollary}

\begin{lemma}\label{lem.halfhomeo}
	Suppose that $f$ is a $C^\alpha$ metric Anosov flow on $Y$. Then
	for any small enough $\epsilon > 0$ there is some $\tau_0 > 0$ such that,
	for any $0< \tau \leq \tau_0$, the map
	\[
		\ang{\,\cdot\,}{\,\cdot\,}^\iota : f_{[-\tau,\tau]}y \times \wuu_\tau(y)
		\to \wu_\epsilon(y)
	\]
	is a bi-H\"older homeomorphism onto a neighborhood of $y$ inside of $\wu_\epsilon(y)$ for any $y
	\in Y$. Furthermore, we have $\ang{f_vy}{x}^\iota = f_vx$ for all $v \in [-\tau,\tau]$ and
	$x \in \wuu_\tau(y)$.
\end{lemma}
\begin{proof}
	Take some $\epsilon > 0$ such that, for any $x,y,z \in Y$, if
	$x \in \wuu_{\epsilon}(y) \cap \ws_{\epsilon}(z)$, then $x = \ang{y}{z}^\iota$.
	It is clear from uniform continuity of $f$ and the definition of canonical coordinates
	that any small enough $\epsilon > 0$ satisfies this.
	We also assume that $\epsilon$ is small enough to satisfy Lemma \ref{lem.subset}.
	Choose $0 < \tau_0 \leq \epsilon$ so that $d(x,y) \leq \tau_0$ implies that $d(f_vx,f_vy) \leq
	\epsilon$ for every $x,y \in Y$ and every $\abs{v} \leq \tau_0$,
	and so that $\tau_0$ is small enough to satisfy Corollary \ref{cor.lpstau}.

	Suppose that $0 < \tau \leq \tau_0$, and let $\abs{v} \leq \tau$ and $x \in \wuu_\tau(y)$.
	Note that
	\[
		f_vx \in \wuu_\epsilon(f_vy) \cap \ws_\epsilon(x)
	,\] 
	and hence $\ang{f_vy}{x}^\iota = f_vx$, which proves the last claim.
	By Corollary \ref{cor.lpstau}, we know that $\ang{\,\cdot\,}{\,\cdot\,}^\iota$ restricts to a
	bi-H\"older homeomorphism of $f_{[-\tau,\tau]}y \times \wuu_\tau(y)$ onto its image, which we have
	now just shown is the set
	\[
		A = \bigcup_{v\in[-\tau,\tau]}f_v \wuu_\tau(y)
	.\]
	Thus we only have to check that $A$ is a neighborhood of $y$ inside of $\wu_\epsilon(y)$.

	Clearly $A \subseteq \wu_\epsilon(y)$ by our assumptions on $\tau_0$ and $\epsilon$. On the
	other hand, if $0 < r \leq \tau$ is another scale chosen so that $d(x,y) \leq r$ implies
	$d(f_vx,f_vy) \leq \tau$ for every $x,y \in Y$ and every $\abs{v} \leq r$,
	we see that $\wu_r(y) \subseteq A$. 
	By Lemma \ref{lem.subset}, $\wu_r(y)$ is a neighborhood of $y$ inside of $\wu_\epsilon(y)$, and
	hence so is $A$.
\end{proof}

By \eqref{eq.HS}, we have the following immediate corollary of Lemma \ref{lem.halfhomeo}.
\begin{corollary}\label{cor.halfhomeo}
	Suppose that $f$ is a $C^\alpha$ metric Anosov flow on $Y$.
	Then for any small enough $\epsilon > 0$ there exists $\tau_0 > 0$ such that,
	for any $0< \tau \leq \tau_0$ and any $y \in Y$, the restricted flow map
	\[
		\FFF: \wuu_\tau(y) \times [-\tau,\tau] \to \wu_\epsilon(y)
	\]
	is a bi-H\"older homeomorphism onto a neighborhood of $y$ inside of $\wu_\epsilon(y)$
\end{corollary}

Combining Corollary \ref{cor.lps} and \ref{cor.halfhomeo}, we obtain now a
complete picture of the local H\"older structure of a $C^\alpha$ metric Anosov flow, allowing us to
prove Theorem \ref{Conjugacy}.

\begin{lemma}\label{lem.homeo}
	For any sufficiently small $\tau > 0$, for any $y \in Y$, the function
	\[
		\Phi_y : \wuu_{\tau}(y)\times [-\tau,\tau] \times \wss_{\tau}(y)  \to U(y,\tau),
		\quad  \Phi_y(x,t,z) = \ang{f_tx}{z} 
	,\]
	is a bi-H\"older homeomorphism onto a neighborhood $U(y,\tau)$ of $y$ in $Y$.
\end{lemma}

\begin{proof}[Proof of Theorem \ref{Conjugacy}]
	Suppose that $\FFF$ and $\FFF'$ are two $C^\alpha$ metric Anosov flows on compact metric spaces $(Y,d)$ and
	$(Y',d')$, respectively, and let $\psi : Y \to Y'$ be a topological conjugacy. We only have to
	show that $\psi$ is H\"older, since the H\"older continuity of $\psi^{-1}$ 
	follows by symmetry.

	Choose constants
	$\eta,\lambda,k > 0$, $C \geq 1$ so that they satisfy the inequalities
	\eqref{eq.HW}, \eqref{eq.HS}, and \eqref{eq.HF} for both flows $\FFF$ and $\FFF'$.
	Let $(\epsilon,\delta)$ be a scale for the canonical coordinates of $\FFF$, and
	suppose $\epsilon' > 0$ is small enough that, for any $x,y,z \in  Y'$, then 
	$x \in \wss_{\epsilon'}(y)\cap\wu_{\epsilon'}(z)$ implies $x = \ang{y}{z}$. 
	We assume that $\epsilon \leq \epsilon'$, and that
	$d(x,z) \leq \epsilon$ implies $d(\psi(x) ,\psi(z)) \leq \epsilon'$ for every $x, z \in Y$.

	Suppose that $y \in Y$, and let $y' = \psi(y)$. Using Lemma \ref{lem.homeo}, choose
	two scales $\tau, \tau' > 0$ so that
	\begin{enumerate}[label=(\alph*)]
		\item $\Phi_{y'}$ is a bi-H\"older homeomorphism onto $U(y',\tau')$,
		\item	$\Phi_y$ is a bi-H\"older homeomorphism onto $U(y,\tau)$,
		\item $d(x,z) \leq 2\tau$ implies $d(\psi(x),\psi(z)) \leq \eta$ for every $x,z \in Y$,
		\item $\psi(U(y,\tau))\subseteq U(y',\tau')$, and
		\item $d(f_tx,z) \leq \delta$ for every 
			$(x,t,z) \in \wuu_\tau(y)\times[-\tau,\tau]\times\wss_\tau(y)$.
	\end{enumerate}
	Suppose that $(x,t,z) \in \wuu_\tau(y)\times[-\tau,\tau]\times\wss_\tau(y)$. Then 
	$d(f_tx,z) \leq \delta$, and hence
	\[
		\Phi_y(x,t,z) = \ang{f_tx}{z}	\in \wss_\epsilon(f'_tx)\cap \wu_\epsilon(z).
	\]
	This implies that
	\[
		\psi(\Phi_y(x,t,z)) \in \wss_{\epsilon'}(f'_t\psi(x))\cap \wu_{\epsilon'}(\psi(z))
	,\]
	and so
	\[
		\psi(\Phi_y(x,t,z)) = \ang{f'_t\psi(x)}{\psi(z)} = \Phi_{y'}(\psi(x),t,\psi(z))
	.\]
	Since $\Phi_y^{-1}$ and $\Phi_{y'}$ are both H\"older, then H\"older continuity of
	$\psi$ on $U(y,\tau)$ follows immediately if we show that $\psi$ is H\"older continuous on
	$\wuu_\tau(y)$ and on $\wss_\tau(y)$.

	We show H\"older continuity of $\psi|_{\wuu_\tau(y)}$.  Suppose that $x,z \in \wuu_\tau(y)$. 
	
	 Let $L=L(2\tau)$ and $T= T(x,z, 2\tau) \geq 0$ be as in \eqref{def:L} and \eqref{def:T} respectively.
	By Lemma \ref{lem.logclose}, 
	\begin{equation} \label{eq.logclose2tau}
	d(f_tx,f_tz) \leq 2\tau \mbox{ for all } \abs{t} \leq T.
	\end{equation}
Then $f_Tx
	\in \wuu_{2\tau}(f_Tz)$, and therefore $f'_T\psi(x) \in \wuu_\eta(f'_T\psi(z))$. By \eqref{eq.HW} and \eqref{relateTandL},
	this implies that we have
	\[
		d(\psi(x),\psi(z)) \leq Ce^{-\lambda T}\eta^\alpha = C \eta^\alpha e^{\lambda L(2\tau)} d(x,z)^{\frac{\lambda\alpha}{k}}
	.\]
	The H\"older continuity of $\psi|_{\wss_\tau(y)}$ is proved analogously.
\end{proof}

The same arguments can be used in the setting of Theorem
	\ref{prop.realstatement}. We assume that the topological conjugacy $\psi$ is uniformly continuous to ensure that it
	preserves the stable and unstable sets of $\FFF$ and $\FFF'$ a priori,
	and so that $\psi$ can be written in terms of its coordinates with respect to $\Phi_y$, $\Phi_{y'}$.
	The arguments in this section then apply without any major modifications to give the following statement.
\begin{proposition}\label{rem.noncompactB}
	Let $\FFF : Y \times \R$ and $\FFF': Y' \times \R$ be continuous flows on metric spaces $(Y,d)$ and
	$(Y',d')$. Suppose that $\FFF$ and $\FFF'$ satisfy the hypotheses of Theorem
	\ref{prop.realstatement}, and suppose that $\psi : Y \to Y'$ is a topological
	conjugacy between $\FFF$ and $\FFF'$ such that $\psi$ and $\psi^{-1}$ are uniformly continuous. 
	Then $\psi$ is \ul{} continuous.
\end{proposition}

\section{Strong Markov coding} \label{s.coding}
In this section, we prove Theorem \ref{Coding}. We accomplish this by verifying the hypothesis of Constantine, Lafont and Thompson from \cite{CLT19} which ensure that there is strong Markov coding. We set up the preliminaries from \cite{CLT19}. These were set up for a metric Anosov flow (in the sense of Pollicott).  We postpone the definition of a metric Anosov flow in the sense of Pollicott until \S \ref{subsec.otherdefinitions} where we will check that every Lipschitz metric Anosov flow is metric Anosov in the sense of Pollicott. Thus, the definitions and results here apply unproblematically using Definition \ref{def.Calpha} with $\alpha=1$.
For a set $D$, and interval $I$, we write
\[ f_I D = \{f_t x : x \in D, \ t \in I\}. \]

A flow $(f_t)$ is \emph{H\"older continuous} if the map from $X \times [0,1] \to X$ given by $(x,t) \to f_t(x)$ is H\"older continuous. It follows that every time$-t$ map is H\"older continuous, and the map $t \to f_t(x)$ is H\"older continuous for each $x\in X$. A flow satisfying \eqref{eq.HF} and the upper inequality in \eqref{eq.HS} is H\"older continuous in this sense.

\begin{definition}\label{defn:section}
	For a continuous flow $( f_t )$ on a metric space $(X,d)$, a \emph{section} is a closed subset $D\subset X$ and a $\xi>0$ so that the map $(z,t) \mapsto f_tz$ is a homeomorphism between $D \times [-\xi, \xi] \to X$ and $f_{[-\xi,\xi]}D$.
\end{definition}

For a section $D \subset X$,  we write $\Int_{\mathcal O} D$ for the interior of $D$ \emph{transverse to the flow}; that is,
\[\Int_\mathcal O D = D \cap \bigcap_{\epsilon>0}(f_{(-\epsilon,\epsilon)}D)^\circ\]
where $Y^\circ$ denotes the interior of $Y$ with respect to the topology of $X$.

For any section $D$, there is a well-defined projection map $\Proj_D: f_{[-\xi,\xi]}D \to D$ defined by $\Proj_D(f_tz) =z$. By definition, the domain of this map contains a nonempty open neighborhood of $X$. 
We recall Bowen's notion of a proper family of sections and a Markov proper family from \cite{rB73}.

\begin{definition}\label{defn:proper}
	Let $\BBB=\{B_1, \ldots, B_n\}$, and $\DDD=\{D_1, \ldots, D_n\}$ be collections of sections. We say that $(\BBB , \DDD)$ is a \emph{proper family at scale $\alpha>0$} if $\{ (B_i, D_i) : i = 1, 2, \ldots, n \}$ satisfies the following properties:
	\begin{enumerate}
		\item $\diam(D_i)< \alpha$ and $B_i \subset D_i$ for each $i \in \{1, 2, \ldots, n\}$;
		\item $\bigcup_{i=1}^nf_{(-\alpha, 0)} (\Int_\mathcal O B_i)=Y$;
		\item For all $i\neq j$, if $f_{[0, 4\alpha]}(D_i) \cap D_j \neq \emptyset$, then $f_{[-4\alpha, 0]}(D_i) \cap D_j = \emptyset$.
	\end{enumerate}
\end{definition}

Condition (3) implies that the sets $D_i$ are pairwise disjoint, and the condition is symmetric under reversal of time; that is, it follows that if $f_{[-4\alpha, 0]}(D_i) \cap D_j \neq \emptyset$, then $f_{[0, 4\alpha]}(D_i) \cap D_j = \emptyset$. We now define a special class of proper families, which we call \emph{pre-Markov}.

\begin{definition}\label{def:rectangle}
	For a metric Anosov flow, a \emph{rectangle} $R$ in a section $D$ is a subset $R \subseteq \Int_\mathcal O D$ such that for all $x,y \in R$, $\Proj_{D} \langle x,y\rangle  \in R$.
\end{definition}

\begin{definition}\label{defn:pollicott}
	Let $(\BBB, \DDD)$ be a proper family at scale $\alpha>0$. We say that $(\BBB, \DDD)$ is \emph{pre-Markov} if the sets $B_i$ are closed rectangles and we have the following property:
	\begin{equation} \label{premarkov}
		\mbox{If }B_i \cap f_{[-2 \alpha, 2 \alpha]}B_j \neq \emptyset, \mbox{ then } B_i \subset f_{[-3\alpha, 3\alpha]} D_j.
	\end{equation}
\end{definition}

	 A Markov proper family is a proper family where the sections are rectangles, and with a property which can be informally stated as `different forward $\RR$-transition implies different future, and different backward $\RR$-transitions implies different past.'  More precisely, a proper family $(\RR, \SSS)$ is \emph{Markov} if the sets $R_i$ are rectangles, and we have the following Markov property:  let $H$ denote the Poincar\'e return map for $\bigcup_i R_i$ with respect to the flow $\{f_t$\}. Then if $ x \in R_i$ and $ H(x) \in R_j$, and $z \in R_i$ and $H(z) \notin R_j$, then $z \notin V^{ss}_{\diam R_i}(x)$. Similarly, if  $ x \in R_i$ and $ H^{-1}(x) \in R_j$, and $z \in R_i$ and $H^{-1}(z) \notin R_j$, then $z \notin V^u_{\diam R_i}(x)$. The existence of a Markov proper family is a key step in the construction of symbolic dynamics. In \cite[Lemma 3.9]{CLT19}, it is proved  (using arguments from \cite{rB73}) that if one has a pre-Markov proper family, there exists a Markov proper family. The Markov proper family is constructed out of small pieces (flowed by a small time)  of the sections in the pre-Markov family. This reduces the task of finding a Markov proper family, which is a main ingredient in the symbolic dynamics construction, to the much easier task of finding a pre-Markov proper family.

Given a collection of sections $\RR$, let $H: \bigcup_{i=1}^N R_i \to \bigcup_{i=1}^N R_i$ be the Poincar\'e (return) map, and let $r: \bigcup_{i=1}^N R_i \to (0, \infty)$ be the return time function, which are well defined in our setting.

\begin{definition}
	For a Markov proper family $(\RR, \SSS)$ for a metric Anosov flow, we define the coding space to be
	\[ \Sigma=\Sigma(\RR)= \left \{ \underline x \in \prod_{-\infty}^{\infty}\{1, 2, \ldots, N\} \mid \text{ for all } l, k \geq 0, \bigcap_{j=-k}^l H^{-j}( \Int_\mathcal O  R_{x_j}) \neq \emptyset \right \}. \]
\end{definition}

In \S2.3 of \cite{Pollicott}, the symbolic space $\Sigma(\RR)$ is shown to be a shift of finite type. There is a canonically defined map $\pi: \Sigma(\RR) \to \bigcup_i R_i$ given by $\pi(\underline x) = \bigcap_{j=-\infty}^\infty H^{-j}(R_{x_j})$. Let $\rho= r \circ \pi: \Sigma \to (0, \infty)$ and let $\Sigma^\rho=\Sigma^\rho(\RR)$ be the suspension flow over $\Sigma$ with roof function $\rho$. We extend $\pi$ to $\Sigma^\rho$ by $\hat \pi(\underline x, t) = f_t(\pi(\underline x))$. Pollicott shows the following.

\begin{theorem} \emph{(}\cite[Theorem 1]{Pollicott}\emph{)}
	If $(f_t)$ is a metric Anosov flow on $Y$, and $(\RR, \SSS)$ is a Markov proper family, then $\Sigma(\RR)$ is a shift of finite type and the map $\hat\pi: \Sigma^\rho \to Y$ is finite-to-one, continuous, surjective, injective on a residual set,  and satisfies  $\hat \pi \circ s_t = f_t \circ \hat \pi$, where $(s_t)$ is the suspension flow on $\Sigma^\rho$.
\end{theorem}

We say that a flow has a \emph{strong Markov coding} if the conclusions of the previous theorem are true with the additional hypothesis that the roof function $\rho$ is H\"older and that the map $\hat \pi$ is H\"older. We have the following criterion for the existence of a strong Markov coding.
\begin{theorem} [\cite{CLT19} Theorem A] \label{t.usefulversion}
	Let $(f_t)$ be a H\"older continuous metric Anosov flow without fixed points. Suppose that there exists a pre-Markov proper family $(\BBB, \DDD)$  satisfying:
	\begin{enumerate}
		\item the return time function $r(y)$ for $\BBB$ is H\"older where it is continuous (i.e., on each $B_i \cap H^{-1}(B_j)$ where $H$ is the Poincar\'e first-return map);

		\item the projection maps along the flow $\Proj_{B_{i}}: B_i \times [-\alpha, \alpha] \to B_i$ are H\"older, where $\alpha>0$ is the scale for the pre-Markov proper family. 
	\end{enumerate}
	Then the flow has a strong Markov coding. 
\end{theorem}
Since the space is compact and the symbolic coding is given by a shift of finite type, transitivity of the flow implies that the strong Markov coding can be chosen to be transitive by restricting to a single irreducible component if necessary. Since metric Anosov flows are expansive, any fixed points are isolated and can thus be removed.

To prove Theorem \ref{Coding}, we show how to verify the hypotheses of Theorem \ref{t.usefulversion} in our setting. Let $(Y,d)$ be a compact metric space, and let $\FFF : Y \to Y$ be a $C^\alpha$ metric Anosov flow
	on $Y$.  We define `standard' rectangles as follows. For each $y \in Y$ and $\tau > 0$ sufficiently small we let
		\begin{equation} \label{standardrec}
			R(y,\tau) \coloneqq \ang{\wuu_\tau(y)}{\wss_\tau(y)}.
		\end{equation}
		 Wecall the homeomorphisms 
	\[
		\Phi_y : \wuu_{\tau}(y)\times [-\tau,\tau] \times \wss_{\tau}(y)  \to U(y,\tau),
		\quad  \Phi_y(x,t,z) = \ang{f_tx}{z} 
	\]
	onto neighborhoods $U(y,\tau)$ that were described in Lemma \ref{lem.homeo}.
	Assuming that $\tau$ is small enough, it is easy to show from the
	definitions that $\Phi_y$ is equivalently defined by $\Phi_y(x,t,z) = f_t\ang{x}{z}$.
\begin{lemma}
For each $y \in Y$ and $\tau > 0$ sufficiently small, then $R(y,\tau)$ is rectangle.
\end{lemma}
\begin{proof}
	From the formula $\Phi_y(x,t,z) = f_t\ang{x}{z}$, it
	immediately follows that 
	\[R(y,\tau) = \Phi_y(\wuu_\tau(y) \times \{0\}\times \wss_\tau(y))\]
	is a continuous section of the flow, such that $f_{[-\tau,\tau]}R(y,\tau) = U(y,\tau)$.

	We see that $R(y,\tau)$ is a rectangle for small enough $\tau$, since the
	Bowen bracket satisfies the axioms of an abstract local bracket operation, and
	hence
	\begin{align*}
		\ang{R(y,\tau)}{R(y,\tau)}
		&= \ang{\ang{\wuu_\tau(y)}{\wss_\tau(y)}}{\ang{\wuu_\tau(y)}{\wss_\tau(y)}} \\
		&= \ang{\wuu_\tau(y)}{\wss_\tau(y)} = R(y,\tau). \qedhere
	\end{align*} 
	\end{proof}
	\begin{lemma} \label{lem.proj}
		Let $R = R(y,\tau)$ be a standard rectangle as in \eqref{standardrec}. Then the projection mapping
		$\mathrm{Proj}_{R} : f_{[-\tau,\tau]}R \to R$ is H\"older continuous.
	\end{lemma}
	\begin{proof} 
		This follows from the fact that $\Phi_y$ is bi-H\"older.
	\end{proof}

	\begin{proposition}\label{prop.holderreturn}
		Suppose that $z_1,z_2 \in Y$ and $f_{T_1}z_1$ and $f_{T_2}z_2$ are both in $R =
		R(y,\tau)$, where $T_1, T_2 \geq 0$ and $\abs{T_1-T_2} \leq \eta$.
		Suppose that $z_1,z_2$ are close enough together so that $f_{T_1}z_2 \in
		f_{[-\tau,\tau]}R = U(y, \tau)$. 
		Then 
		\[
			\abs{T_1-T_2} \leq Kd(z_1,z_2)^\gamma
		.\]
	\end{proposition}
	\begin{proof}
		By \eqref{eq.HF} we have 
		\[
			d(f_{T_1}z_1,f_{T_1}z_2) \leq Ce^{kT_1}d(z_1,z_2)^\alpha
		.\] 
		On the other hand, since the projection map is H\"older, we have
		\[
			d(f_{T_1}z_1,f_{T_2}z_2) \leq C_0 d(f_{T_1}z_1,f_{T_1}z_2)^{\beta} \leq
			C_1 d(z_1,z_2)^{\alpha\beta}
		.\] 
		Thus
		\[
			d(f_{T_1}z_2, f_{T_2}z_2) \leq d(f_{T_1}z_2,f_{T_1}z_1) + d(f_{T_1}z_1,f_{T_2}z_2) \leq
			C_2d(z_1,z_2)^{\alpha\beta}
		.\] 
		But since $\abs{T_1-T_2} \leq \eta$, then by \eqref{eq.HS}, we must have
		\[
			C^{-1}\abs{T_1-T_2}^{\frac{1}{\alpha}}
			\leq d(f_{T_1}z_2,f_{T_2}z_2) \leq C_2d(z_1,z_2)^{\alpha\beta}. \qedhere
		\] 
	\end{proof}
The above lemma can be applied when $z_1,z_2$ are in a domain of continuity for a return map to a union of sections, where $T_1,T_2$ are the return times. When these sections are constructed from standard rectangles, Theorem \ref{t.usefulversion} gives the H\"older return times in (1) of Proposition \ref{prop.holderreturn}. 
\begin{remark}
We note that Proposition \ref{prop.holderreturn} does not use any particular properties about $R(y,\tau)$ except that it has a H\"older projection map. The proof relies on the lower bound in
		\eqref{eq.HS} which we did not assume in \cite{CLT19}, which is why property (1) is not automatic from property (2) in the setup of Theorem \ref{t.usefulversion}.
\end{remark}

Now suppose that $\FFF : Y \to Y$ is a Lipschitz metric Anosov flow on $Y$. It suffices to work with this case because we show in Lemma \ref{lem.pollicott} that a $C^\alpha$ metric Anosov flow is H\"older equivalent to a Lipschitz metric Anosov flow and that this is a metric Anosov flow in the sense of Pollicott and \cite{CLT19}.

We follow the argument of \cite{CLT19}. The following lemma is proved analogously to the construction in Proposition 4.10 of \cite{CLT19}.
\begin{lemma}
	There exists a proper family $(\BBB, \DDD)$  at scale $\alpha>0$  where the $B_i$ and $D_i$ are `standard' rectangles of the form $R$ given in Definition \ref{standardrec}.
\end{lemma}
We then apply the following result, which is Lemma 4.11 of \cite{CLT19} to a proper family of `standard' rectangles.
\begin{lemma} \label{markovpf}
	Let $\FFF$ be a Lipschitz continuous expansive flow on a compact metric space. Given a proper family $(\BBB, \DDD)$ for $\FFF$ at scale $\alpha>0$ where the $B_i$ and $D_i$ are rectangles, there exists a pre-Markov proper family $(\BBB', \DDD')$ at scale $\alpha>0$ such that every $D_k'\in\mathcal{D}'$ is the image under $f_{s_k}$ of some $D_i\in\mathcal{D}$ where $|s_k| << \alpha$.
\end{lemma}

We now have a proper family $(\BBB', \DDD')$ which is pre-Markov and such that every $D_k'\in\mathcal{D}'$ is the image under $f_{s_k}$ of a standard rectangle. It follows from Lemma \ref{lem.proj} and Proposition \ref{prop.holderreturn} that (1) and (2) of Theorem \ref{t.usefulversion} are verified, and thus the flow admits a strong Markov coding.
When the flow is transitive, the results for equilibrium states stated in Theorem \ref{Coding} follow from the existence of the strong Markov coding, see \cite[\S 6]{CLT19}.

\section{Geodesic flows} \label{s.geodesicflow}
We develop our applications to geodesic flows in negative curvature.
\subsection{CAT($-1$) spaces}\label{subsec.CAT-1}
Let $(X, d)$ be a proper $\CAT(-1)$ space.  
We set up notation and collect some facts we will use. For more extensive background, references
include \cite{Roblin, wB95, BH99}. We study the space of geodesics
\[
	GX \coloneqq \{c : \R \to X \colon c \text{ is an isometric embedding}\}
.\]

The space $GX$ carries a natural geodesic flow 
\[g_t : GX \to GX \qquad (g_tc)(s) = c(s+t)\].
We equip $G  X$ with the metric
\[ d_{G  X}(  c,  c') :=  \int_{-\infty}^\infty d( c(s),  c'(s)) e^{-2|s|}ds.\]

We let $\partial X$ be the boundary at infinity of $X$, defined to be
the collection of equivalence classes of geodesic rays, where two rays are equivalent if they stay within
bounded distance of each other. We equip $\partial X$ with the cone topology, which makes $\partial X$
compact. Given a geodesic $c$, then $c(\infty)$ and $c(-\infty)$ are used to denote the
points in $\partial  X$ corresponding to the positive and negative geodesic rays defined by $c$,
respectively. The union $\bar X = X\sqcup\partial X$ can be given a topology which restricts to the cone
topology on $\partial X$ and compactifies $X$ in such a way that
$\lim_{t\to\pm\infty}c(t) = c(\pm\infty)$ for any $c \in GX$.

On $GX$ we define the dynamical sets $W^\epsilon(c)$ for $\epsilon \in \{\mathrm{uu},\mathrm{u},
\mathrm{ss},\mathrm{s}\}$ by setting
\[
	W^\mathrm{uu}(c) = \{c' \in GX \colon \lim_{t\to-\infty}d_{GX}(g_tc,g_tc') = 0\},
	\quad W^\mathrm{u}(c) = \bigcup_{t \in \R}g_t W^{\mathrm{uu}}(c)
,\] 
and similarly
\[
	W^\mathrm{ss}(c) = \{c' \in GX \colon \lim_{t\to\infty}d_{GX}(g_tc,g_tc') = 0\},
	\quad W^\mathrm{s}(c) = \bigcup_{t \in \R}g_t W^{\mathrm{ss}}(c)
.\] 
Each of these induces a partition on $GX$ which we denote by $W^\sigma$ for
$\sigma \in \{\mathrm{uu},\mathrm{u},\mathrm{ss},\mathrm{s}\}$. T he dynamical partitions can be characterized using the asymptotic geometry as follows. 
We have 
\[
	W^\mathrm{u}(c) = \{c' \in GX \colon c'(-\infty) = c(-\infty)\}, \quad
	W^\mathrm{s}(c) = \{c' \in GX \colon c'(\infty) = c(\infty)\}
.\] 
We define the Busemann functions by
\[
	\beta_\xi(x,y) = \lim_{z \to \xi} (d(x, z)-d(y,z))
\]
for $\xi \in \bX$, $x, y \in X$. 
The level sets of the Busemann functions $\beta_\xi(\cdot,y)$ are called the \emph{horospheres}
centered at $\xi$, and we write $H_{\xi, y} = \{ x : \beta_\xi(x,y)=0\}$.
Within each weak stable (resp. unstable) set, we then have the strong stable (resp. unstable) set given by
\begin{align*}
	\wuu(c) &= \{c' \in \wu(c) \colon c'(0) \in H_{c(-\infty),c(0)}\},\\
	\wss(c) &= \{c' \in \ws(c) \colon c'(0) \in H_{c(\infty),c(0)}\}.
\end{align*}
It is well known from the negative curvature geometry that the families $\wuu$ and $\wss$ satisfy \eqref{eq.HW} globally. 

There is an identification of $GX$ with $\partial^2 X \times \R$. These are called \emph{Hopf cooordinates}. This identifies the orbit of each $c \in GX$
with $\{(c(-\infty),c(\infty))\}\times \R$ and conjugates 
$g_t$ to the flow on the $\R$ component. This identification is not unique, but if
$\varphi,\varphi' : GX \to \partial^2 X \times \R$ are two such identifications, then there is $\kappa : \partial^2 X \to \R$ so that
\[
	(\varphi' \circ \varphi^{-1}) (\xi,\eta,t) = (\xi,\eta,t+\kappa(\xi,\eta)).
\]  

We have homeomorphisms
\[
	\begin{tikzcd}
		W^\mathrm{uu}(c) \arrow[rr, "c' \mapsto c'(\infty)"] &  &
		\partial X \setminus \{c(-\infty)\},
	\end{tikzcd}
	\quad
	\begin{tikzcd}
		W^\mathrm{ss}(c)\arrow[rr, "c' \mapsto c'(-\infty)"] &  &
		\partial X \setminus \{c(\infty)\}.
	\end{tikzcd}
\]
It follows that, provided $c'(\infty) \neq c(-\infty)$,
then $\wss(c) \cap \wu(c')$ contains exactly one geodesic.

For any $c,c' \in GX$ such that $c'(\infty) \neq c(-\infty)$, the product operation $\ang{\:\,}{\:}$
	\[
		\{\ang{c}{c'}\} = \wss(c)\cap \wu(c')
	\] 
is thus well-defined. In contrast to Definition \ref{def.cc}, the bracket operation is defined globally. Of course, by definition, there exists $v=v(c,c')$ so that
\[
\ang{c}{c'} \in \wuu(g_v c').
\]

To show that the bracket operation gives us canonical coordinates, it remains to show that for $K_1>0$ and $d(c, c') \leq K_1$, there exists $K_2$ so that $|v(c,c')| \leq K_2$ and $\ang{c}{c'} \in \wss_{K_2}(c)\wuu_{K_2}(g_v c')$. Using \eqref{eq.HW}, we obtain canonical coordinates at all scales. We leave the details of this argument as an exercise, and we note that further exposition was given in \cite[\S3]{CLT19}.

Since $g_t$ is unit speed it obviously satisfies \eqref{eq.HS}, and it is easily verified that 
\eqref{eq.HF} is satisfied, since
\begin{align*}
	d_{GX}(g_tc,g_tc') &= \int_{-\infty}^{\infty}d(c(s+t),c'(s+t))e^{-2\abs{s}}\,ds\\
										 &= \int_{-\infty}^{\infty}d(c(s),c'(s))e^{-2\abs{s-t}} \leq e^{2\abs{t}}d_{GX}(c,c').
\end{align*}
We conclude that the geodesic flow on GX satisfies the hypotheses of Theorem \ref{prop.realstatement}\footnote{Thus, if $X/ \Gamma$ is compact, then the geodesic flow on $GX/\Gamma$ is $C^\alpha$ metric Anosov. See Remark \ref{rem:limitiations} for the limitations of directly applying statement Theorem \ref{prop.realstatement} when $X/ \Gamma$ is non-compact.} It follows from Theorem \ref{prop.realstatement} that the bracket operation of the flow is \ul{} continuous on $(GX \times GX)_\delta$ for $\delta >0$ and that the holonomy maps are H\"older with uniform constants and exponents. See the statement of Corollary \ref{GeodesicFlow}.
\subsection{Pinched negative curvature manifolds and H\"older holonomies} 
We now consider the case when $X$ is a pinched negatively curved Hadamard manifold. Then $X$ is a $\CAT(-\kappa^2)$ space for some $\kappa>0$. We can identify the unit tangent bundle $T^1X$ with the space of geodesics $GX$. We can thus equip $T^1X$ with the metic $d_{GX}$. In this metric, the geodesic flow on $GX$ satisfies the hypotheses of Theorem \ref{prop.realstatement}. There are other relevant metrics in addition to $d_{GX}$: 
\begin{itemize}
\item The Sasaki metric $d_S$ is the natural Riemannian metric on $T^1X$.
 \item The Knieper-Manning metric $d_1$ on $GX$ (and thus $T^1X$) is:
  \[
d_1(c,c')
    := \max_{t \in [0,1]}
       d_X\bigl(c(t),c'(t)\bigr).
\]
\item The metric $d'_{GX}$ on geodesics which use  $e^{-t^2}$ for normalization:
\[
d'_{GX}(c,c')
=
\frac{1}{\sqrt{\pi}}
\int_{\mathbb{R}}
d_X\bigl(c(t),c'(t)\bigr)e^{-t^2}\,dt.
\]
\end{itemize}
In Paulin-Pollicott-Schapira, the metric $d'_{GX}$ is used instead of $d_{GX}$. In pinched negative curvature, they show in \cite[Lemma 2.3]{PPS} that $d'_{GX}$ is \ul{}-equivalent to $d_1$. Similarly, one can show that $d_{GX}$ is \ul{}-equivalent to $d_1$ by adapting the proof of \cite[Lemma 2.3]{PPS}. In pinched negative curvature, $d_1$ is Lipschitz equivalent to $d_S$.   Thus, for a pinched negative curvature Hadamard manifold all the metrics $d_1, d_S, d_{GX}, d'_{GX}$ are \ul{}-equivalent. Perhaps the main advantage of using $d'_{GX}$ rather than $d_{GX}$ is that there exists a constant $C>0$ so that $d'_{GX}(c, c')\leq Cd_1(c, c')$. 

Now let $\Gamma$ be a non-elementary torsion-free Kleinian group and let $X_0 = X/\Gamma$. Since these metrics are $\Gamma$-invariant, they induce metrics on the quotient $GX/\Gamma$ which is identified with the unit tangent bundle of $X_0$. Clearly, the \ul{}-equivalence between these metrics on $GX$  passes to the metrics on $GX/\Gamma$ and thus $T^1X_0$. 
In summary, we have the following result.
\begin{proposition} \label{oneholderclass}
	Let $X_0 = X / \Gamma$ be a pinched negative curvature manifold. Then $d_{GX}$, $d_1$, and $d_S$ are all \ul{}
	equivalent considered on $GX_0$ or $T^1X_0$. 
\end{proposition}
We emphasize that the curvature pinching assumption is essential for these equivalences. For $\CAT(-1)$ spaces, Proposition \ref{oneholderclass} no longer holds. For example, if $X$ is a tree, $d_{GX}$ and $d'_{GX}$ are not \ul{}-equivalent.

The detailed definition of a foliation $\mathcal{W}$ of a manifold having H\"older holonomies is given in Pugh, Shub, and Wilkinson \cite[p.520]{PSW}.  We take arbitrary local transversals to $\mathcal{W}$ and slide along the leaves of $\mathcal{W}$, equipping the transversals with a metric inherited from a fixed metric. The definition asks that the H\"older constants should be uniform over all `reasonable' transversals. The class of reasonable transversals is defined in detail in \cite[p.521]{PSW}.

By Corollary \ref{Holderholonomy}, the $u$-holonomies are
H\"older continuous in $d_{GX}$ when the transversals are local leaves of $\wss$. By Proposition
\ref{oneholderclass}, the same is true if we fix one of the other natural metrics one can consider in this
setting. To obtain H\"older holonomies in the sense of the above definition, we must bootstrap this regularity from the case where the transversals
are $\wss$ leaves to the case where the transversals are as above. We thus prove Theorem
\ref{Holonomies} by proving the following proposition.
\begin{proposition} \label{thmD:detailed}
Let $(M,\mathfrak{g})$ be a complete connected pinched negatively curved manifold.  If the $u$-holonomy maps on $\wgss$ and the
	$ss$-holonomy maps on $\wguu$ are H\"older, then $\wgu$ and $\wgss$ have H\"older holonomies. Similarly, $\wguu$ and $\wgs$ have H\"older holonomies. 
\end{proposition}
The following arguments suffice to prove the proposition. Suppose $M^n$ is a smooth manifold with transverse foliations $W^1$ and $W^2$ tangent to
	continuous distributions $E^1$ and $E^2$, of dimensions $n_1$ and $n_2$,
	respectively. Suppose the local product map for these two foliations is H\"older continuous.

	Let $p \in M$ be arbitrary. Because $E^1$ is a continuous distribution, we can choose
	smooth coordinates sending $p$ to $0 \in \R^n$ so that $E^1$ is
	``almost horizontal" in a small enough neighborhood around $0$. That is, we write $\R^n = \R^{n_1}\times \R^{n_2}$ and for $x \in
	\R^n$ we write $x = (x_1,x_2)$ for $x_i \in \R^{n_i}$. Then there are
	smooth coordinates such that, for a small enough
	neighborhood around $p$ and any $q$ in this neighborhood, if $v \in E^1_q$, then
	$v$ has slope at most $\kappa$. That is, we have
	\[
		\abs{v_2} \leq \kappa \abs{v_1},
	\]
	and we can assume $\kappa < 1$. For simplicity, we extend everything to all of $\R^n$.
	We now consider the following setup: we are given foliations $W^1$, $W^2$ of $\R^n$ tangent to
	continuous distributions $E^1$, $E^2$ such that $E^1$ is almost horizontal.
	Any two leaves of $W^1$, $W^2$ intersect each other exactly once, and they give rise
	to a (locally) H\"older product function $\ang{\:\,}{\:}:(\R^n)^2 \to \R^n$ given by
	\[
		\{\ang{x}{y}\} = W^1(x) \cap W^2(y)
	.\] 
	Suppose we take any arbitrary transversal $T$ to $W^1$,
	and suppose that $T$ is ``almost vertical" in the
	sense that all the \emph{coslopes} along $T$ are at most $\kappa$.
	That is, for any vector $v \in \R^n$ tangent to $T$, we have $\abs{v_1} \leq \kappa\abs{v_2}$.
	
	Since $\kappa < 1$ then the slopes of $T$ are all strictly greater than any of the slopes of $W^1$. Let $x,y \in \R^n$ be arbitrary, and let $x', y' \in T$ be the images under the $W^1$-projection
	to $T$. We show that $d(x',y') \leq
	Cd(x,y)^\alpha$ for some $C,\alpha$.

	\begin{lemma}\label{lem.transverse}
		Since $T$ is almost vertical and $W^1$ is almost horizonal, then
		there is a constant $\kappa'$ such that
		\[
			d(x',\ang{x'}{y'}) \leq \kappa' d(y', \ang{x'}{y'})
		.\] 
	\end{lemma}
	\begin{proof}[Sketch Proof]
		If $d(y',\ang{x'}{y'})$ is small compared
		with $d(x',\ang{x'}{y'})$, then the vector from $x'$ to $\ang{x'}{y'}$ is similar to the
		vector from $x'$ to  $y'$. However, the first vector must have slope at most  $\kappa < 1$ and the
		second vector must have coslope at most $\kappa < 1$, which is impossible.
	\end{proof} 
	Using Lemma \ref{lem.transverse}, H\"older continuity of the $W^1$-projection to $T$ follows readily:
	\begin{align*}
		d(x',y') & \leq d(x',\ang{x'}{y'}) + d(\ang{x'}{y'},y') \\
						 &\leq (1+\kappa')d(\ang{x'}{y'}, y')\\
						 &= (1+\kappa')d(\ang{x}{y'},\ang{y}{y'}) \quad (\text{since $x \in W^1(x')$, $y \in W^1(y')$})\\
						 &\leq (1+\kappa')C d(x,y)^{\alpha}.
	\end{align*}

It follows from these arguments that Proposition \ref{thmD:detailed} holds. It follows from Corollary \ref{GeodesicFlow} and Proposition \ref{thmD:detailed} that Theorem \ref{Holonomies} holds.
\begin{remark}
This result has particular interest for the Ballmann-Brin-Burns examples. We have shown that the $B^3$ examples have H\"older holonomies, even though it was shown in \cite{BBB} that they do not have H\"older tangent distributions. We recall that \cite[Theorem 4]{HW} says that if the leaves of a foliation $\WWW$ are uniformly $C^2$ and $\WWW$ has H\"older continuous holonomies, then the tangent distribution to $\WWW$ is H\"older. We thus obtain Corollary \ref{cor:B3}: the strong stable foliations $\wgss$ for the $B^3$ example can not be uniformly $C^2$.
\end{remark}

\section{Relationships between definitions of metric Anosov flows}\label{subsec.otherdefinitions}
We offer an alternative formulation of $C^\alpha$ metric Anosov flows in terms of laminations, and we compare with the original definition of Pollicott.
\subsection{Definition using laminations}
Suppose that $Y$ is a compact metric space.
Consider a collection of homeomorphisms
\[
	\mathcal{A} = \{\phi_U = (\phi^1_U,\phi^2_U) : U \to K_U \times L_U\}_{U \in \,\mathcal{U}}
,\]
where $\mathcal{U}$ is an open cover of $Y$, and $L_U$ and $K_U$ are topological spaces that may
depend on $U$. We say that $\mathcal{A}$ is a 
\emph{lamination atlas} if, for any $U,V \in \,\mathcal{U}$, and any $x,y \in U \cap V$, we
have
\[
\phi^1_U(x) = \phi^1_U(y) \iff \phi^1_V(x) = \phi^1_V(y)
.\]
We call each $\phi_U$ a \emph{chart} of the atlas $\mathcal{A}$. 
The level sets of each $\phi^1_U$, which are homeomorphic to $L_U$,
are called \emph{plaques} of the atlas, and the collection of plaques of a lamination atlas is
called a \emph{plaquation}.
The definition of a lamination atlas thus asks that any plaque
inside some $U \in \mathcal{U}$ intersects at most one plaque in another $V \in \mathcal{U}$.
If $\mathcal{A}$ is a lamination atlas, and $x,y$ are in the domain $U$ of a single chart
$\phi_U \in \mathcal{A}$, then we write $x \sim_{\mathcal{A}} y$ if and only if
$\phi^1_U(x) = \phi^1_U(y)$, that is, when $x$ and $y$ are in the same plaque of 
$\mathcal{A}$ inside $U$. This is independent of the choice of $U \in \mathcal{U}$ containing both
$x$ and $y$. We leave the statement $x\sim_{\mathcal{A}} y$ undefined if $x,y$ do not lie in any
single chart of $\mathcal{A}$. Thus we think of $\sim_{\mathcal{A}}$ as a truth function taking
values in $\{0,1\}$ with domain $\{(x,y) \in Y\times Y : x,y \in U = \mathrm{dom}(\phi_U) \text{ for
some } \phi_U \in \mathcal{A}\}$.

Let us call two lamination atlases $\mathcal{A}_1, \mathcal{A}_2$
\emph{compatible} if $\mathcal{A}_1 \cup \mathcal{A}_2$ is a lamination atlas. Note that
$\mathcal{A}_1$ and $\mathcal{A}_2$ are compatible if and only if $\sim_{\mathcal{A}_1}$ and
$\sim_{\mathcal{A}_2}$ agree on the intersection of their domains. Compatibility is not
necessarily an equivalence relation on lamination atlases.
To obtain a true equivalence relation, we need to pass to refinements, in the following
sense. Given open covers 
$\{L^i_U\}_{i\in I_U}$ of $L_U$ and  $\{K^j_U\}_{j \in J_U}$ of $K_U$, and letting
$U_{ij} = \phi_U^{-1}(K^i_U \times L^j_U)$, then
\[
	\mathcal{A}' = \{\left.\phi_{U}\right|_{U_{ij}} \colon U \in \mathcal{U},\, i\in I_U,\, j\in J_U\}
\]
is called a \emph{refinement} of $\mathcal{A}$. Any refinement of a lamination
atlas is a lamination atlas.  We say that $\mathcal{A}_1$ and $\mathcal{A}_2$
are \emph{equivalent} if they have refinements $\mathcal{A}_1'$ and $\mathcal{A}_2'$ such that
$\mathcal{A}_1'$ and $\mathcal{A}_2'$ are compatible. Then 
two atlases $\mathcal{A}_1$ and $\mathcal{A}_2$ are equivalent 
if and only if $\sim_{\mathcal{A}_1}$ and $\sim_{\mathcal{A}_2}$ agree on a neighborhood of
the diagonal of $Y\times Y$, which may be smaller than the intersection of their domains.
This defines an equivalence relation on lamination atlases, which we define to be
a \emph{lamination} $\mathcal{L}$.

Suppose that $\mathcal{L}$ is a lamination, and let $\mathcal{A}$ be an atlas of
$\mathcal{L}$. A \emph{plaque chain} is a sequence of plaques $P_1,\ldots, P_n$ of $\mathcal{A}$
such that $P_i \cap P_{i+1} \neq \emptyset$.
We can associate to $\mathcal{L}$ a global partition $\mathcal{F}$,
whose elements would be called the \emph{leaves} of the lamination,
such that $x,y$ are in the same element of $\mathcal{F}$ if
and only if $x \in P_1$ and $y \in P_n$ for some plaque chain $P_1,\ldots, P_n$. A priori, this
definition of $\mathcal{F}$ depends on the choice of atlas $\mathcal{A}$ of $\mathcal{L}$.
Let us denote it $\mathcal{F}(\mathcal{A})$.
Say that an atlas $\mathcal{A}$ \emph{has connected plaques} if
every plaque of $\mathcal{A}$ is a connected subset of $Y$. If two equivalent atlases $\mathcal{A}$,
$\mathcal{A}'$ both have connected plaques,
then it is possible to show that $\mathcal{F}(\mathcal{A}) = \mathcal{F}(\mathcal{A}')$.
Thus if a lamination $\mathcal{L}$ has at least one atlas with connected plaques,
then $\mathcal{F}$ can be defined consistently by setting $\mathcal{F} =
\mathcal{F}(\mathcal{A})$ for any such choice of $\mathcal{A}$. 
\begin{remark}
Often the plaques are not locally connected or even totally disconnected; then it is generally not 
possible to define any global structure from the lamination alone. For
example, let $Y = \{0,1\}^\Z$, let $\sigma : Y \to Y$ be the shift map, and consider a plaquation of
$Y$ by local stable sets of $\sigma$. Due to the ultrametric property,
$\mathcal{F}(\mathcal{A)}$ is just the partition of
$Y$ into local stable sets whose diameter at each point is the maximum diameter of the plaques of
$\mathcal{A}$ passing through that point. This depends on the particular plaquation used,
and it does not match the expectation that $\mathcal{F}(\mathcal{A})$ is the partition of $Y$ into
global stable sets.
\end{remark}
\begin{remark}
If a lamination $\mathcal{L}$ is contracted by a homeomorphism $f$
at some scale,  one obtains a well-defined global partition $\mathcal{F}$ 
via the equivalence relation
\begin{equation}\label{eq.globalrelation}
	x\, \mathcal{F}\, y \iff f^nx \sim_\mathcal{A} f^ny \text{ for all large enough $n$},
\end{equation}
where $\mathcal{A}$ is any atlas of $\mathcal{L}$. This global partition will depend on the map $f$ in general.
For example, consider $Y = \{0,1\}^\Z$ again with left
shift $\sigma : Y \to Y$, and let $\mathcal{L}$ be the stable lamination of $\sigma$.
Then we do indeed obtain the
usual partition of $Y$ into global stable sets of $\sigma$ by setting $f = \sigma$ in
\eqref{eq.globalrelation}. On the other hand, let $\psi : Y \to Y$ be the homeomorphism
\[
	\psi(x)_i = \begin{cases}
		x_i,& i \leq 0, \\
		x_i+x_0\:\:(\mathrm{mod}\,2),& i\geq 1.
	\end{cases}
\]
Then $\psi$ preserves the stable lamination $\mathcal{L}$,
but not the partition into global stable sets.
The map $\sigma' = \psi^{-1} \circ \sigma \circ \psi$
preserves and contracts $\mathcal{L}$ at a small enough scale,
but it gives rise to a different global partition: indeed, the points
\[
	y = \cdots 00\underline{0}00\cdots, \quad y' = \cdots 00\underline{1}11\cdots
\]
are in the same global stable set of $\sigma'$ but not $\sigma$.
\end{remark}

Given an atlas $\mathcal{A}$ of a lamination $\mathcal{L}$, it is useful to define
\[
	\mathcal{L}(y,\epsilon;\mathcal{A}) \coloneqq \{x \in Y \colon x\sim_\mathcal{A} y \text{ and }
	d(x,y) < \epsilon\}
\]
for each $y \in Y$ and $\epsilon > 0$. Then
$\mathcal{L}(y,\epsilon;\mathcal{A})$ is an open neighborhood of $y$ inside
the union of plaques of $\mathcal{A}$ containing $y$, and if $\epsilon$ is small then
$\mathcal{L}(y,\epsilon;\mathcal{A})$ is an open neighborhood of $y$ inside a single plaque of
$\mathcal{A}$. If $\mathcal{A}_1$ and $\mathcal{A}_2$ are two atlases of the same lamination
$\mathcal{L}$, then there is some $\epsilon > 0$ such
that $\mathcal{L}(y,\epsilon';\mathcal{A}_1) = \mathcal{L}(y,\epsilon';\mathcal{A}_2)$ for all $0 <
\epsilon' \leq \epsilon$ and all $y \in Y$.

We say that a pair of laminations $(\mathcal{L},\mathcal{K})$ is a \emph{transverse pair}
if there is some atlas $\mathcal{A}$ that is simultaneously an atlas for $\mathcal{L}$ and
$\mathcal{K}$, in the sense that $\mathcal{A}$ is a lamination atlas for
$\mathcal{L}$, and the collection $\{(\phi^2_U,\phi_U^1) \colon (\phi_U^1, \phi_U^2)
\in \mathcal{A}\}_{U\in\mathcal{U}}$ is a lamination atlas for $\mathcal{K}$.
Suppose that $\mathcal{A}$ is one such atlas for a transverse pair $(\mathcal{L},\mathcal{K})$.
If $\delta > 0$ is a Lebesgue number of the cover $\mathcal{U}$,
then we obtain a map
\[
	[\,\cdot\,, \,\cdot \,]: (Y \times Y)_\delta \to Y, \qquad [x,y] =
		\phi_U^{-1}(\phi^1_U(x),\phi^2_U(y))
,\]
where $x,y \in U \in \mathcal{U}$.
It is easy to check that this definition is independent of $U$,
and $[x,y]$ is the unique intersection point of an $\mathcal{L}$-plaque through $x$
and a $\mathcal{K}$-plaque through $y$:
\[
	\{[x,y]\} = (\phi^1_U)^{-1}(\phi^1_U(x)) \cap (\phi^2_U)^{-1}(\phi^2_U(y))
.\]
After possibly halving $\delta$, then it can be shown that this map is a local product structure on
$(Y\times Y)_\delta$ in the sense of Definition \ref{def.lps}.

Conversely, suppose that $[\,\cdot\, ,\,\cdot\,]$
is an (abstract) local bracket operation. Then we are able to recover a transverse pair of laminations
$(\mathcal{L},\mathcal{K})$ in the following way. Given $y \in Y$, $\epsilon > 0$, recall the
definitions
\begin{align*}
L_\epsilon(y) &= \{x \in Y \colon x = [x,y] \text{ and } d(x,y) < \epsilon\}, \\
	R_\epsilon(y) &= \{x \in Y \colon x = [y,x] \text{ and } d(x,y) < \epsilon\}.
\end{align*}
If $\epsilon > 0$ is small enough, then
the restriction of the bracket operation to $L_\epsilon(y) \times R_\epsilon(y)$ 
is a homeomorphism onto an open neighborhood $U_{y,\epsilon}$ of $y$ for each $y \in Y$
by Lemma \ref{lem.ruelle}.
Let $\phi_y : U_{y,\epsilon} \to L_\epsilon(y)\times R_\epsilon(y)$ denote the inverse of
this homeomorphism. Possibly shrinking $\epsilon > 0$, then it is straightforward to show from the
local bracket structure definition that
the collection 
\[
	\mathcal{A} = \{\phi_y : U_{y,\epsilon} \to L_\epsilon(y)\times
	R_\epsilon(y)\}_{y \in Y}
\]
is a simultaneous atlas for two laminations $\mathcal{L}$ and
$\mathcal{K}$, which are therefore a transverse pair. For every $0 < \epsilon' \leq \epsilon$,
we have
\[
	L(y,\epsilon') = \mathcal{L}(y,\epsilon';\mathcal{A}), \quad
	K(y,\epsilon') = \mathcal{K}(y,\epsilon';\mathcal{A}) 
.\]
Thus if $\mathcal{A}_\mathcal{L}, \mathcal{A}_\mathcal{K}$ are any other lamination
atlases of $\mathcal{L},\mathcal{K}$, then
\[
	L(y,\epsilon') = \mathcal{L}(y,\epsilon';\mathcal{A}_\mathcal{L}), \quad
	K(y,\epsilon') = \mathcal{K}(y,\epsilon';\mathcal{A}_\mathcal{K})
\]
for every small enough $\epsilon' > 0$. Combining this fact with Lemma \ref{lem.lpscc}, we obtain the
following corresponding lemma.
\begin{lemma}\label{lem.lamcc}
	Suppose that $(\mathcal{L},\mathcal{K})$ is a transverse pair corresponding to a local bracket operation $[\,\cdot\, ,\,\cdot\,]$, and suppose that
	$\mathcal{A}_\mathcal{L}$ and $\mathcal{A}_\mathcal{K}$ are atlases of $\mathcal{L}$ and
	$\mathcal{K}$, respectively. Then for any small enough $\epsilon > 0$, there is some $\delta > 0$
	such that, if $d(x,y) \leq \delta$, then
	\[
		\{[x,y]\} = \mathcal{L}(y,\epsilon;\mathcal{A}_\mathcal{L}) \cap
		\mathcal{K}(y,\epsilon;\mathcal{A}_\mathcal{K})
	.\]
\end{lemma}

Suppose now that $\FFF : Y \times \R \to Y$ is a continuous flow on $Y$. A lamination $\mathcal{L}$ is \emph{invariant} under $\FFF$
if there is some atlas $\mathcal{A}$ of $\mathcal{L}$ and some $t_0 > 0$
such that $\mathcal{A}$ is compatible with $f_t^*\mathcal{A}$ for every $\abs{t} \leq t_0$, where
$f_t^*\mathcal{A} \coloneqq \{\phi_U \circ f_t \colon \phi_U \in \mathcal{A}\}$.
If $\mathcal{A}$ is one such atlas, then it is easy to see that
for any $\epsilon > 0$, there is some $\epsilon' > 0$ such
that
\[
	f_t\mathcal{L}(y,\epsilon';\mathcal{A}) \subseteq \mathcal{L}(f_ty,\epsilon;\mathcal{A})
\]
for every $\abs{t} \leq t_0$.
Hence, for any other atlas $\mathcal{A}_\mathcal{L}$ of an invariant lamination $\mathcal{L}$, for
any $\epsilon > 0$, there is some $\epsilon' > 0$ such that 
\begin{equation}\label{eq.invariantlam}
	f_t\mathcal{L}(y,\epsilon';\mathcal{A}_\mathcal{L}) \subseteq
	\mathcal{L}(f_ty,\epsilon;\mathcal{A}_\mathcal{L})
\end{equation}
for every $\abs{t} \leq t_0$.
Using this fact, we see that a local bracket operation associated
to invariant laminations must behave well under the flow, in the following way.
\begin{lemma}\label{lem.invariantlps}
	Let $\FFF : Y\times \R$ be a continuous flow, and suppose that
	$[\,\cdot\, ,\,\cdot\,]$ is a local bracket operation on $Y$ corresponding to a transverse pair
	$(\mathcal{L},\mathcal{K})$ such that $\mathcal{L}$ and $\mathcal{K}$  are invariant under the
	flow. Then there is some scale $\delta > 0$ such that, for any $x,y \in Y$ and any $t \geq 0$, then
	\begin{gather*}
		\max_{0\leq s\leq t}d(f_sx,f_sy) \leq \delta \implies f_t[x,y] = [f_tx,f_ty],\\
		\max_{-t\leq s\leq 0}d(f_sx,f_sy) \leq \delta \implies f_{-t}[x,y] = [f_{-t}x,f_{-t}y].
	\end{gather*}
\end{lemma}
\begin{proof}
	We show that there is some $t_0 > 0$ such that the lemma holds assuming $0 \leq t \leq t_0$. The statement extends to all $t \geq 0$ by an inductive argument.

	Let $\mathcal{A}_\mathcal{L}, \mathcal{A}_\mathcal{K}$ be any atlases of $\mathcal{L}$,
	$\mathcal{K}$. Using Lemma \ref{lem.lamcc}, choose
	$\epsilon, \delta > 0$ so that
	$\{[x,y]\} = \mathcal{L}(x,\epsilon;\mathcal{A}_\mathcal{L}) \cap
	\mathcal{K}(y,\epsilon;\mathcal{A}_\mathcal{K})$ whenever $d(x,y) \leq \delta$.
	Using \eqref{eq.invariantlam}, we let $0 < \epsilon' \leq \epsilon$ and $t_0 > 0$ 
	be small enough so that
	\[
		f_t\mathcal{L}(y,\epsilon';\mathcal{A}_\mathcal{L}) \subseteq
		\mathcal{L}(f_ty,\epsilon;\mathcal{A}_\mathcal{L}), \quad
		f_t\mathcal{K}(y,\epsilon';\mathcal{A}_\mathcal{K}) \subseteq
		\mathcal{K}(f_ty,\epsilon;\mathcal{A}_\mathcal{K})
	\]
	for every $\abs{t} \leq t_0$ and $y \in Y$. Possibly shrinking $\delta$ and using continuity of
	the local product structure, we additionally assume that for any $x,y \in Y$,
	\[
		d(x,y) \leq \delta \implies d(x,[x,y]) < \epsilon' \text{ and } d(y,[x,y]) < \epsilon'.
	\]

	Suppose that $0 \leq t \leq t_0$, and let $x,y \in Y$ satisfy 
	$\max_{0\leq s \leq t}d(f_sx,f_sy) \leq \delta$. Then
	\[
		[x,y] \in
		\mathcal{L}(x,\epsilon;\mathcal{A}_\mathcal{L}) \cap
		\mathcal{K}(y,\epsilon;\mathcal{A}_\mathcal{K})
	.\]
	Since $d(x,[x,y])$ and $d(y,[x,y])$ are both less than $\epsilon'$, then
	\[
		[x,y] \in
		\mathcal{L}(x,\epsilon';\mathcal{A}_\mathcal{L}) \cap
		\mathcal{K}(y,\epsilon';\mathcal{A}_\mathcal{K})
	,\]
	and hence
	\[
		f_t[x,y] \in \mathcal{L}(f_tx,\epsilon;\mathcal{A}_\mathcal{L}) \cap
		\mathcal{K}(f_ty,\epsilon;\mathcal{A}_\mathcal{K})
	.\]
	Therefore $f_t[x,y] = [f_tx,f_ty]$. The proof of the second statement of the lemma is
	then similar.
\end{proof}

Given an invariant lamination $\mathcal{L}$, we say that $\mathcal{L}$ is
\emph{transverse to the flow} if $\mathcal{L}$ has an atlas of the form
\[
\mathcal{A} = \{\phi_U : U \to (K_U \times (-\epsilon_U,\epsilon_U ))\times L_U\}_{U \in \mathcal{U}},
\]
such that $\phi^{-1}_U(k,t+s,l) = f_s(\phi^{-1}_U(k,t,l))$ whenever 
$t,t+s\in (-\epsilon_U,\epsilon_U)$. Note that this definition requires $f$ to have no fixed points, and
after assuming without loss of generality that each $\epsilon_U$ is small enough,
then $\mathcal{A}$ is also a lamination atlas for the lamination of $Y$ by $f$-orbits.
Furthermore, then $\mathcal{A}$ serves as an atlas for a lamination
$\mathcal{L}^\mathrm{c}$ which saturates $\mathcal{L}$ with the flow direction. That is,
$\mathcal{L}^\mathrm{c}$ is defined to be the lamination with atlas $\mathcal{A}$ and plaques of
the form $\phi_U^{-1}(\{k_0\} \times (-\epsilon_U,\epsilon_U) \times L_U)$ for 
$U \in \mathcal{U}$, $k_0 \in K_U$. It is clear from the definition of $\mathcal{A}$ that, for every
$\epsilon > 0$, there is some $\epsilon' > 0$ such that
\[
\mathcal{L}^\mathrm{c}(y,\epsilon';\mathcal{A}) 
\subseteq \bigcup_{\abs{t}\leq\epsilon}f_t\mathcal{L}(y,\epsilon;\mathcal{A}), \quad
\bigcup_{\abs{t}\leq\epsilon'}f_t\mathcal{L}(y,\epsilon';\mathcal{A})
\subseteq \mathcal{L}^\mathrm{c}(y,\epsilon;\mathcal{A})
\]
for each $y \in Y$.
Hence, given any atlases $\mathcal{A}_\mathcal{L},\mathcal{A}_{\mathcal{L}^\mathrm{c}}$ of
$\mathcal{L},\mathcal{L}^\mathrm{c}$, then for every $\epsilon > 0$, there is some
$\epsilon' > 0$ such that for every $y \in Y$,
\begin{equation}\label{eq.center}
	\mathcal{L}^\mathrm{c}(y,\epsilon';\mathcal{A}_{\mathcal{L}^\mathrm{c}}) 
	\subseteq \bigcup_{\abs{t}\leq\epsilon}f_t\mathcal{L}(y,\epsilon;\mathcal{A}_{\mathcal{L}}),
	\quad \bigcup_{\abs{t}\leq\epsilon'}f_t\mathcal{L}(y,\epsilon';\mathcal{A}_{\mathcal{L}})
	\subseteq \mathcal{L}^\mathrm{c}(y,\epsilon;\mathcal{A}_{\mathcal{L}^\mathrm{c}}).
\end{equation}

Using all of the terminology we have introduced so far, we give an equivalent 
definition of $C^\alpha$ metric Anosov flows similar to the one in \cite{BCLS}.

\begin{lemma}
	Let $\FFF : Y \times \R \to Y$ be a continuous flow on a compact metric space. Then $\FFF$ is a
	$C^\alpha$ (resp.\ Lipschitz) metric Anosov flow if and only if $\FFF$ satisfies \eqref{eq.HF} and
	\eqref{eq.HS} with exponent $\alpha$ (resp.\ exponent $\alpha = 1$)
	and there are invariant laminations $\mathcal{W}^\mathrm{ss}$, $\mathcal{W}^\mathrm{uu}$ such that
	\begin{itemize}
		\item $\mathcal{W}^\mathrm{uu}$ and $\mathcal{W}^{\mathrm{ss}}$ are transverse to the flow.
		\item $(\mathcal{W}^\mathrm{uu},\mathcal{W}^\mathrm{s})$ is a transverse pair, where
			$\mathcal{W}^\mathrm{s} = (\mathcal{W}^\mathrm{ss})^\mathrm{c}$.
		\item $(\mathcal{W}^\mathrm{ss},\mathcal{W}^\mathrm{u})$ is a transverse pair, where
			$\mathcal{W}^\mathrm{u} = (\mathcal{W}^\mathrm{uu})^\mathrm{c}$.
		\item There exist $C,\lambda > 0$ and a plaquation $\mathcal{P}^\mathrm{uu}$ of
			$\mathcal{W}^\mathrm{uu}$ such that, for any $x,y \in P \in
			\mathcal{P}^{\mathrm{uu}}$ and any $t \geq 0$, then
			\[
				d(f_{-t}x,f_{-t}y) \leq Ce^{-\lambda t}d(x,y)^\alpha
			.\]
		\item There exist $C, \lambda > 0$ and a plaquation $\mathcal{P}^\mathrm{ss}$
			of $\mathcal{W}^\mathrm{ss}$ such that, for any ${x,y \in P \in \mathcal{P}^{\mathrm{ss}}}$ and
			any $t \geq 0$, then
			\[
				d(f_tx,f_ty) \leq Ce^{-\lambda t}d(x,y)^\alpha
			.\]
	\end{itemize}
\end{lemma}

\begin{proof}
	If $f$ is a $C^\alpha$ metric Anosov flow, then it satisfies \eqref{eq.HF} and \eqref{eq.HS} by
	definition. The Bowen brackets $\ang{\,\cdot\,}{\,\cdot\,}$ and $\ang{\,\cdot\,}{\,\cdot\,}^\iota$
	are local product structures by Lemma \ref{lem.lps}, and hence they give rise to transverse pairs
	of laminations, $(\mathcal{W}^\mathrm{ss},\mathcal{W}^\mathrm{u})$ and
	$(\mathcal{W}^\mathrm{uu},\mathcal{W}^\mathrm{s})$. 
	For each $\ell \in \{\mathrm{uu},\mathrm{ss},\mathrm{u},\mathrm{s}\}$, then the (sufficiently small)
	plaques of $\mathcal{W}^{\ell}$ are open subsets of the sets $W^\ell_\epsilon(y;d)$ for $y \in Y$,
	where $\epsilon > 0$ is sufficiently small. Thus it is 
	immediate from \eqref{eq.HW} and Lemma \ref{lem.subset} that there are sufficiently fine
	plaquations $\mathcal{P}^\mathrm{uu}, \mathcal{P}^{\mathrm{ss}}$ of
	$\mathcal{W}^\mathrm{uu},\mathcal{W}^\mathrm{ss}$ that satisfy the required H\"older contraction
	estimates of the lemma.
	
	The inverses of the homeomorphisms from Lemma \ref{lem.homeo} can be used to show
	that $\mathcal{W}^\mathrm{ss}$ is transverse to the flow with $\mathcal{W}^\mathrm{s} =
	(\mathcal{W}^\mathrm{ss})^\mathrm{c}$, and by a symmetrical argument, $\mathcal{W}^\mathrm{uu}$ is
	transverse to the flow with $\mathcal{W}^\mathrm{u} = (\mathcal{W}^\mathrm{uu})^\mathrm{c}$. This
	proves the forward implication.

	In the other direction, suppose that $f$ satisfies \eqref{eq.HF} and \eqref{eq.HS} with exponent
	$\alpha$, and suppose that there are laminations $\mathcal{W}^\mathrm{ss},
	\mathcal{W}^\mathrm{uu}$ satisfying the conditions in the statement of the lemma.
	To show that $f$ is a $C^\alpha$ metric Anosov flow, we have to check that $f$ satisfies
	\eqref{eq.HW} and that $f$ has canonical coordinates.

	Fix atlases $\mathcal{A}^\mathrm{ss},\mathcal{A}^\mathrm{uu}$ of $\mathcal{W}^\mathrm{ss},
	\mathcal{W}^\mathrm{uu}$. First, we show that for every $\epsilon > 0$, there is some
	$\epsilon' > 0$ such that
	\begin{equation}\label{eq.inclusions1}
		\begin{split}
			\mathcal{W}^\mathrm{ss}(y,\epsilon';\mathcal{A}^\mathrm{ss}) \subseteq \wss_\epsilon(y; d), \quad
			\mathcal{W}^\mathrm{uu}(y,\epsilon';\mathcal{A}^\mathrm{uu}) \subseteq \wuu_\epsilon(y; d), \\
			\wss_{\epsilon'}(y;d) \subseteq \mathcal{W}^\mathrm{ss}(y,\epsilon;\mathcal{A}^\mathrm{ss}), \quad
			\wuu_{\epsilon'}(y;d) \subseteq \mathcal{W}^\mathrm{uu}(y,\epsilon;\mathcal{A}^\mathrm{uu})
		\end{split}
	\end{equation}
	for every $y \in Y$.
  By assumption, there exist $C,\lambda,\eta > 0$ such
	that, if $x \in \mathcal{W}^\mathrm{ss}(y,\eta\,;\mathcal{A}^\mathrm{ss})$ and $x' \in
	\mathcal{W}^\mathrm{uu}(y,\eta\,;\mathcal{A}^\mathrm{uu})$, then
	\begin{equation}\label{eq.shrinking}
		d(f_{t}x, f_{t}y)\leq Ce^{-\lambda t}d(x,y)^\alpha, \quad
		d(f_{-t}x', f_{-t}y)\leq Ce^{-\lambda t}d(x',y)^\alpha \quad
	\end{equation}
	for all $t \geq 0$. 
	Thus, given $\epsilon > 0$,
	the first two inclusions of \eqref{eq.inclusions1} are obvious after choosing 
	$0 < \epsilon' \leq \eta$ such that, for any $x,y \in Y$, 
	\[
		d(x,y) < \epsilon' \implies \max_{t \in [-S,S]}d(f_tx,f_ty) \leq \epsilon, \quad \text{where }
		S = -\tfrac{1}{\lambda}\log \tfrac{\epsilon}{C}
	.\]

	Now we prove the second two inclusions.
	Let $\mathcal{W}^\mathrm{u} = (\mathcal{W}^\mathrm{uu})^\mathrm{c}$ and $\mathcal{W}^\mathrm{s} =
	(\mathcal{W}^\mathrm{ss})^\mathrm{c}$, and fix atlases
	$\mathcal{A}^\mathrm{u},\mathcal{A}^\mathrm{s}$ of
	$\mathcal{W}^\mathrm{u},\mathcal{W}^\mathrm{s}$. By assumption, the pairs
	$(\mathcal{W}^\mathrm{ss},\mathcal{W}^\mathrm{u})$ and
	$(\mathcal{W}^\mathrm{uu},\mathcal{W}^\mathrm{s})$
	are transverse, and so they define local product structures $[\,\cdot\, ,\,\cdot\,]$ and
	$[\,\cdot\, ,\,\cdot\,]^\iota$, respectively.
	Suppose that $\epsilon > 0$. Since $f$ has no fixed points, we can choose $\tau > 0$
	such that, for any $0 < v \leq \tau$, there is some $r > 0$ such that $d(f_vy,y) \geq r$ for every
	$y \in Y$. Using \eqref{eq.center}, choose $\tau' \leq \min\{\epsilon,\eta\}$ such  that
	\[
		\mathcal{W}^\mathrm{u}(y,\tau';\mathcal{A}^\mathrm{u}) \subseteq \bigcup_{\abs{v}\leq
		\tau}f_v\mathcal{W}^\mathrm{uu}(y,\tau;\mathcal{A}^\mathrm{uu})
	\]
	for any $y \in Y$. Since $\mathcal{W}^\mathrm{ss}$ and $\mathcal{W}^\mathrm{u}$ 
	are invariant, we can choose a scale $\delta > 0$ from 
	Lemma \ref{lem.invariantlps} for the local product structure $[\,\cdot\, ,\,\cdot\,]$.
	Using Lemma \ref{lem.lamcc}, let  $0 < \epsilon' \leq \delta$ be
	small enough so that $d(x,y) \leq \epsilon'$ implies that, for any $x,y \in Y$, we have
	\[
		[x,y] \in \mathcal{W}^\mathrm{ss}(x,\tau';\mathcal{A}^\mathrm{ss})
		\cap\mathcal{W}^\mathrm{u}(y,\tau';\mathcal{A}^\mathrm{u}).
	\]

	Suppose that $y' \in \wss_{\epsilon'}(y;d)$. Let $z = [y,y']$.
	By Lemma \ref{lem.invariantlps}, then for each $t \geq 0$ we have $f_tz = [f_ty,f_ty']$, and 
	in particular,
	\[
		f_tz \in \mathcal{W}^\mathrm{ss}(f_ty,\tau';\mathcal{A}^\mathrm{ss})
		\cap \mathcal{W}^\mathrm{u}(f_ty',\tau';\mathcal{A}^\mathrm{u})
	.\]
	For each $t \geq 0$, we can choose some $\abs{v_t} \leq \tau$ such that
	\[
		f_{t+v_t}z \in \mathcal{W}^\mathrm{uu}(f_ty',\tau;\mathcal{A}^\mathrm{uu})
	,\]
	and by \eqref{eq.shrinking}, we have
	\[
		d(f_{v_t}z, y') \leq Ce^{-\lambda t}\tau^\alpha
	\]
	for every $t \geq 0$. Passing to a subsequence, we conclude that there is some
	$\abs{v} \leq \tau$ such that $f_vz = y'$. But since $z \in
	\mathcal{W}^\mathrm{ss}(y,\tau';\mathcal{A}^\mathrm{ss})$, then by \eqref{eq.shrinking}, we have
	$\lim_{t\to\infty}d(f_tz,f_ty) = 0$.
	By assumption we also have $\lim_{t\to\infty}d(f_ty,f_ty') = 0$, and
	hence $\lim_{t\to\infty}d(f_tz,f_ty') = \lim_{t\to\infty} d(f_tz,f_{t+v}z) = 0$. Since
	$\abs{v}\leq \tau$, then this implies that $v = 0$. Therefore 
	\[
		y' = z \in \mathcal{W}^\mathrm{ss}(y,\tau';\mathcal{A}^\mathrm{ss})
		\subseteq \mathcal{W}^\mathrm{ss}(y,\epsilon;\mathcal{A}^\mathrm{ss})
	.\]
	Since $y'$ was arbitrary, we have shown that
	\[
		\wss_{\epsilon'}(y;d) \subseteq \mathcal{W}^\mathrm{ss}(y,\epsilon;\mathcal{A}^\mathrm{ss})
	.\]
	The proof of the last inclusion in \eqref{eq.inclusions1} 
	then follows by a symmetric argument.

	By \eqref{eq.inclusions1} and \eqref{eq.shrinking}, then it is clear that $f$ satisfies
	\eqref{eq.HW}. Using \eqref{eq.invariantlam}, \eqref{eq.center}, and \eqref{eq.inclusions1}  we
	also see that, for any $\epsilon > 0$, there is some $\epsilon' > 0$ such that
	\begin{equation}\label{eq.inclusions2}
		\begin{split}
			\mathcal{W}^\mathrm{s}(y,\epsilon';\mathcal{A}^\mathrm{s}) \subseteq \ws_\epsilon(y; d), \quad
			\mathcal{W}^\mathrm{u}(y,\epsilon';\mathcal{A}^\mathrm{u}) \subseteq \wu_\epsilon(y; d), \\
			\ws_{\epsilon'}(y;d) \subseteq \mathcal{W}^\mathrm{s}(y,\epsilon;\mathcal{A}^\mathrm{s}), \quad
			\wu_{\epsilon'}(y;d) \subseteq \mathcal{W}^\mathrm{u}(y,\epsilon;\mathcal{A}^\mathrm{u})
		\end{split}
	\end{equation}
	for every $y \in Y$. Hence, by Lemma \ref{lem.lpscc}, \eqref{eq.inclusions1}, and
	\eqref{eq.inclusions2}, then $f$ must have canonical coordinates as well. This proves that 
	$f$ is a $C^\alpha$ metric Anosov flow.
\end{proof}

\subsection{Metric Anosov flows in the sense of Pollicott} We recall the original definition of metric Anosov flow, in the equivalent form provided by \cite{BCLS}.
\begin{definition}
	Let $(Y,d)$ be a compact metric space.
	We say that a continuous flow $\FFF : Y \times \R \to Y$ is a \emph{metric Anosov flow (in the sense of
	Pollicott)} if there exist invariant laminations $\mathcal{W}^\mathrm{ss}$,
	$\mathcal{W}^\mathrm{uu}$ such that
	\begin{itemize}
		\item $\mathcal{W}^\mathrm{uu}$ and $\mathcal{W}^{\mathrm{ss}}$ are transverse to the flow.
		\item $(\mathcal{W}^\mathrm{uu},\mathcal{W}^\mathrm{s})$ is a transverse pair, where
			$\mathcal{W}^\mathrm{s} = (\mathcal{W}^\mathrm{ss})^\mathrm{c}$.
		\item $(\mathcal{W}^\mathrm{ss},\mathcal{W}^\mathrm{u})$ is a transverse pair, where
			$\mathcal{W}^\mathrm{u} = (\mathcal{W}^\mathrm{uu})^\mathrm{c}$.
		\item There exist $C,\lambda > 0$ and a plaquation $\mathcal{P}^\mathrm{uu}$ of
			$\mathcal{W}^\mathrm{uu}$ such that, for any $x,y \in P \in
			\mathcal{P}^{\mathrm{uu}}$ and any $t \geq 0$, then
			\[
				d(f_{-t}x,f_{-t}y) \leq Ce^{-\lambda t}d(x,y)
			.\]
		\item There exist $C, \lambda > 0$ and a plaquation $\mathcal{P}^\mathrm{ss}$
			of $\mathcal{W}^\mathrm{ss}$ such that, for any ${x,y \in P \in \mathcal{P}^{\mathrm{ss}}}$ and
			any $t \geq 0$, then
			\[
				d(f_tx,f_ty) \leq Ce^{-\lambda t}d(x,y)
			.\]
	\end{itemize}
\end{definition}

\begin{corollary}\label{cor.comparison}
	If $\FFF$ is a metric Anosov flow in the sense of Pollicott, and $\FFF$ satisfies \eqref{eq.HF} and
	\eqref{eq.HS} with exponent $\alpha \in (0,1)$ (resp.\ $\alpha = 1$), then $\FFF$ is a $C^\alpha$
	(resp.\ Lipschitz) metric Anosov flow in the sense of Definition \ref{def.Calpha}.
	Conversely, if $\FFF$ is a $C^\alpha$ Anosov flow,
	and $\FFF$ satisfies \eqref{eq.HW} with exponent $\alpha = 1$, then $\FFF$ is
	a metric Anosov flow in the sense of Pollicott. In particular, any Lipschitz metric Anosov flow in
	our sense satisfies Pollicott's definition.
\end{corollary}

\begin{remark}\label{rem.weaker}
	Some authors use the following weaker version of \eqref{eq.HW}:
	\begin{equation}\label{eq.HW0}
		d(f_tx,f_ty) \leq Ce^{-\lambda t}, \quad d(f_{-t}x',f_{-t}y) \leq Ce^{-\lambda t}
		\tag{HW\textsubscript{0}}
	\end{equation}
	whenever $x \in \wss_\eta(y)$, $x' \in \wuu_\eta(y)$, and $t \geq 0$.
	Assuming that the flow satisfies \eqref{eq.HF}, then it is straightforward to check that
	\eqref{eq.HW0} implies \eqref{eq.HW}. 

	To see this, suppose that $x \in \wss_\eta(y)$.
	Then $f_{-T}x \in \wss_\eta(f_{-T}y)$,
	where $T = \max\{-\frac{\alpha}{k}\log d(x,y)+\frac{1}{k}\log\eta,0\}$, so for each $t \geq 0$ we have
	\[
		d(f_tx,f_ty) \leq Ce^{-\lambda(t+T)} \leq C \eta^{\frac{\lambda}{k}}
		e^{-\lambda t}d(x,y)^{\frac{\lambda\alpha}{k}}
	,\] 
	and similarly for $x' \in \wuu_\eta(y)$.
\end{remark}

\begin{lemma}\label{lem.pollicott}
	Let $\FFF$ is a $C^\alpha$ metric Anosov flow on $(Y,d)$. There is a H\"older equivalent metric $d'$ so that $\FFF$ is metric Anosov on $(Y, d')$ in the sense of Pollicott. \end{lemma}
\begin{proof}
	Let $\eta,\lambda,k > 0$, $C \geq 1$ be the constants from \eqref{eq.HW}, \eqref{eq.HS}, and
	\eqref{eq.HF} with respect to the metric $d$.  Choose any $0 < \lambda' < \lambda$, and let
	\[
		d'(x,y) = \int_{-\infty}^{\infty} d(f_tx,f_ty)e^{-\lambda'\abs{t}}\,dt
	.\]

	It is easy to check that that $d'$ is a metric. To see that it is H\"older equivalent to $d$,
	we use \eqref{eq.HF} to conclude that there is some $C' > 0$ such that
	\[
		C'd(x,y)^{\frac{1}{\alpha}}
		\leq \int_{-1}^1 d(f_tx,f_ty)e^{-\lambda'\abs{t}}\,dt 
		\leq d'(x,y).
	\]
	On the other hand, let $D$ be the diameter of  $Y$ with respect to $d$.
	If $d(x,y) \leq 1$ and $S = -\tau\log d(x,y) \geq 0$, then there is some $C' > 0$
	such that
	\begin{align*}
		d'(x,y) 
		&\leq \int_{-S}^{S} Ce^{kS} d(x,y)^\alpha\,dt + \int_{\abs{t}\geq S}De^{-\lambda'\abs{t}}\,dt\\
		&\leq C'Se^{kS}d(x,y)^\alpha + C'e^{-\lambda'S} \\
		&= C'(-\tau \log d(x,y))d(x,y)^{\alpha-k\tau}+C'd(x,y)^{\lambda'\tau}.
	\end{align*}
	Choosing $\tau \in (0,\frac{\alpha}{k})$, then the above exponents are
	both positive.

	By an easy calculation we see that $d'$ now satisfies the bi-Lipschitz property
	\begin{align*}
		d'(f_tx,f_ty)\leq e^{\lambda'\abs{t}}d'(x,y)
	\end{align*}
	for all $x,y \in Y$ and $t \in \R$, and it still satisfies the H\"older speed condition with the
	same H\"older exponent,
	\[
		\frac{1}{C'}\abs{t}^{\frac{1}{\alpha}}\leq d'(y,f_ty) \leq C'\abs{t}^\alpha
	\]
	for all $y \in Y$ and $\abs{t} \leq \eta$, where $C' > 0$ is some constant.
	Furthermore, if $x \in \wss_{\eta}(y)$ and $t \geq 0$, then
	\begin{gather*}
		d'(f_tx,f_ty) = \int_{-\infty}^\infty d(f_sx,f_sy)e^{-\lambda'\abs{s-t}}\,ds\\
									\leq \int_{-\infty}^{0}De^{-\lambda'(t-s)}\,ds 
									+ \int_{0}^{t} Ce^{-\lambda s}\eta^\alpha e^{-\lambda'(t-s)}\,ds
									+ \int_{t}^{\infty} Ce^{-\lambda s}\eta^\alpha e^{-\lambda'(s-t)}\,ds \\
									= e^{-\lambda't}\left(\int_{-\infty}^{0} De^{\lambda's}\,ds \right)
									+ \frac{C\eta^\alpha}{\lambda-\lambda'}(e^{-\lambda' t}-e^{-\lambda t})
									+ e^{-\lambda t}
									\left(\int_{t}^{\infty}C\eta^\alpha e^{-(\lambda'+\lambda)(s-t)}\,ds\right).
	\end{gather*}
	Doing a similar calculation for $x' \in \wuu_\eta(y)$, we conclude that
	there is some $C' > 0$ such that
	\[
		d'(f_tx,f_ty) \leq C'e^{-\lambda' t}, \quad d'(f_{-t}x',f_{-t}y) \leq C'e^{-\lambda't}
	.\]
	Therefore, with respect to $d'$, the flow satisfies \eqref{eq.HF} with growth rate $\lambda'$ and
	H\"older exponent $1$, and it satisfies \eqref{eq.HW0} with contraction rate $\lambda'$. Applying 
	the calculation from Remark \ref{rem.weaker}, we find that the flow 
	satisfies \eqref{eq.HW} with
	exponent $\alpha = \frac{\lambda'\cdot 1}{\lambda'} = 1$ with respect to $d'$.
	Hence by Corollary \ref{cor.comparison}, the metric $d'$ makes
	$\FFF$ a metric Anosov flow in the sense of Pollicott.
\end{proof}

\begin{figure}[!htbp]
  \centering
\begin{tikzpicture}[
    dot/.style={circle, fill, inner sep=1.5pt},
    heading/.style={align=center, anchor=south},
    region/.style={align=center, font=\scriptsize, text width=2.1cm}
  ]
  \fill[gray!20] ( 1.5,0) circle (2.6);
  \fill[gray!8]  (-1.5,0) circle (2.6);
  \draw (-1.5,0) circle (2.6);
  \draw ( 1.5,0) circle (2.6);

  \node[heading] at (-2.9, 2.75)
    {$(X,d,\mathcal{F})$ is $C^{\alpha}$\\metric Anosov};
  \node[heading] at ( 2.9, 2.75)
    {$(X,d,\mathcal{F})$ metric Anosov\\in the sense of Pollicott};

  \node[font=\footnotesize] at (-1.5, 1.5) {Strong Markov coding exists};

  \node[region] at (2.7, 0) {No strong Markov coding guaranteed --- want to exclude};

  \node[dot, label=below:{$(X,d,\mathcal{F})$}]  (dl) at (-2.6,0) {};
  \node[dot, label=below:{$(X,d',\mathcal{F})$}] (dr) at ( 0,0) {};

  \draw[dotted, -{Stealth[length=2.5mm]}, shorten <=3pt, shorten >=3pt]
    (dl) to[bend left=25]
    node[above, pos=0.35, font=\footnotesize] {$\exists\, d'\sim d$} (dr);
\end{tikzpicture}
\caption{Relationship between metric Anosov flow definitions and the existence of a strong Markov coding.}
  \label{fig:metric-anosov}
\end{figure}

\begin{remark}
	If $\FFF$ is a $C^\alpha$ metric Anosov flow, and $\FFF$ satisfies \eqref{eq.HS} with H\"older exponent
	$1$, then the metric $d'$ from the proof of Lemma \ref{lem.pollicott} 
	makes $f$ a Lipschitz metric Anosov flow for which $k=\lambda$.
	In this metric, the Bowen bracket attains the maximum
	H\"older exponents that can be guaranteed by our proof of Theorem \ref{MainTheorem}
	(see Remark \ref{rem.optimal}).
	For a general $C^\alpha$ metric Anosov flow,
	it is not clear if it is always possible to replace
	the metric so that $\FFF$ satisfies \eqref{eq.HS} with H\"older exponent $1$.

\end{remark}

\begin{remark}
In Figure \ref{fig:metric-anosov}, there are no known examples in the set on the right. Although we do not expect that any `natural examples' live there, we do not believe that it is the empty set. It may be possible to shine more light on this issue by constructing an example which is a metric Anosov flow in the sense of Pollicott but which is not $C^\alpha$. We do not pursue that here since our interest in such flows is only to exclude them from consideration. 
\end{remark}
\subsection*{Acknowledgements.} The first named author would like to thank Alexandro Luna for helpful discussions. The second-named author used AI to draw Figure \ref{fig:ss-holonomy} and Figure \ref{fig:metric-anosov} in this paper. No AI tools were used for proofs or writing.

\bibliographystyle{siam} 
\bibliography{biblioHolder-Jun2026}
\end{document}